\documentclass{article}
\usepackage{indentfirst}
\usepackage{amssymb}
\usepackage{amsmath,mathtools}
\usepackage{amsthm,verbatim}
\usepackage{tikz,tikz-cd}
\usetikzlibrary{matrix,arrows,decorations.pathmorphing}
\usepackage{graphicx}
\usepackage[linesnumbered,ruled]{algorithm2e}
\usepackage{amsrefs}
\usepackage[top=1 truein,bottom=1.5 truein,hmargin={1 truein,1 truein}]{geometry}
\usepackage[all,cmtip]{xy}

\SetKwRepeat{Do}{do}{while}

\makeatletter
\newtheorem*{rep@theorem}{\rep@title}
\newcommand{\newreptheorem}[2]{%
\newenvironment{rep#1}[1]{%
 \def\rep@title{#2 \ref{##1}}%
 \begin{rep@theorem}}%
 {\end{rep@theorem}}}
\makeatother

\theoremstyle{plain}
\newtheorem{theorem}{Theorem}
\theoremstyle{definition}
\newtheorem{definition}[theorem]{Definition}
\newtheorem*{remark}{Remark}
\newtheorem{lemma}[theorem]{Lemma}
\newtheorem{proposition}[theorem]{Proposition}
\newtheorem{corollary}[theorem]{Corollary}
\newtheorem{example}[theorem]{Example}
\newtheorem{conjecture}[theorem]{Conjecture}
\newreptheorem{theorem}{Theorem}
\newreptheorem{conjecture}{Conjecture}
\newreptheorem{proposition}{Proposition}

\newcommand{\N}{\mathbb{N}}
\newcommand{\Z}{\mathbb{Z}}
\newcommand{\Q}{\mathbb{Q}}
\newcommand{\R}{\mathbb{R}}
\newcommand{\C}{\mathbb{C}}
\newcommand{\F}{\mathbb{F}}
\newcommand{\Aff}{\mathbb{A}}
\newcommand{\Proj}{\mathbb{P}}
\newcommand{\A}{\alpha}

\newcommand{\e}{\equiv}

\newcommand{\la}{\langle}
\newcommand{\ra}{\rangle}

\newcommand{\inverselimit}{\varprojlim_{F^{ad}}\Aff^{N,ad}_K\cup\{\infty\}}
\newcommand{\inverselimitb}{\varprojlim_{F^{\flat,ad}}\Aff^{N,ad}_{K^\flat}\cup\{\infty\}}
\newcommand{\orbit}[2]{\mathcal{O}_{#1}(#2)}

\DeclareMathOperator{\red}{red}

\DeclareMathOperator{\Per}{Per}
\DeclareMathOperator{\Spec}{Spec}
\DeclareMathOperator{\Spa}{Spa}
\DeclareMathOperator{\Gal}{Gal}

\title{Unlikely intersection problems for restricted lifts of $p$-th power}
\author{Wayne Peng}
\date{\today}

\begin{document}

\maketitle	
\begin{abstract}
I define a morphism $\phi:\Aff^n_{\C_p}\to\Aff^n_{\C_p}$ called a lift of $p$-th power if its natural restriction to the residue field of $\C_p$ is a $p$-th power of some morphism. This definition generalizes from the lift of Frobenius. In this paper the theory of perfectoid spaces is applied on the morphism to the Manin-Mumford conjecutre, the Mordell-Lang conjecture and the Tate-Volch conjecture.
\end{abstract}
\subsection*{Acknowledgement}
I have to thank Junie Xie for introducing this topic in Workshop on interactions between Model Theory and Arithmetic Dynamics in Field Institute in Toronto in 2016. My work generalizes his result in~\cite{Xie}.
\section{Introduction}
A canonical idea from algebraic geometry states that if the intersection of a curve $C$ with a countable collection of ``special” points $\{P_i\}_{i\geq 1}$ is unlikely to be large unless $C$ itself is also special in some way. Questions base on this idea are called unlikely intersection question. Ihara, Serre and Tate are pioneers on this type of problems. For algebraic dynamical systems, one can ask, for example, suppose a curve $C\subset X$ contains infinitely many periodic points under a morphism $F:X\to X$, can we say $C$ is periodic under $F$? I am going to present three questions of unlikely intersection for algebraic dynamical systems, and apply Scholze's Perfectoid theory to these questions. 

Let $\C_p$ be the complete field of the algebraic closed field of field of $p$-adic numbers $\overline{\Q_p}$. I will work on a special kind of endomorphism on $\Aff_{\C_p}^N$, called the \emph{restricted lift of $p$-th power}. The lift of $p$-th power is a kind of endomorphisms generalizing the settings of lift of Frobenius. Since the definition of such an endomorphism is slightly technical I will only define lift of $p$-th power and leave the restrict condition in later section (see Section~\ref{section:3}).

Writing $F$ coordinate-wise as $(F_1,F_2,\ldots,F_N)$, I assume that $F_i$s are polynomials with coefficients defined in $\C_p^\circ$, the valuation ring of $\C_p$. We say $F$ is a lift of $p$-th power if there exists $G\in\C_p^\circ[X_1,\ldots, X_N]$ and $\varpi\in\C_p^\circ$ with $|p|\leq|\varpi|<1$ for the value $|\cdot|$ on $\C_p$ such that
\[
F_i=(G_i)^p\mod \varpi\C_p^\circ.
\]
For example, $(X^p+Y^p+pXY,X^p-Y^p+pXY)$ is a lift of $p$-th power on $\Aff_{\C_p}^2$. Note this endomorphism of $\Aff_{\C_p}^2$ is also a restricted of $p$-th power.

The theory of perfectoid spaces by Scholze is one of celebrating breakthrough in math in 2010s. I use this theory on three questions, the dynamical Manin-Mumford conjecture, the dynamical Mordell-Lang conjecture and  the dynamical Tate-Volach conjecture. Let me state them one by one in the following with my results.
\subsection*{Manin-Mumford Conjecture}

What Manin-Mumford Conjecture tries to say is very close to the idea of unlikely intersections. The version of dynamical systems probably is easier to state than its original version. \cite{zhang1995} gives
\begin{conjecture}\label{conjecture:manin-mumford}
Let $F:X\to X$ be an endomorphism of a quasi-projective variety defined over $K$. Let $V$ be a subvariety of $X$. If the Zariski closure of the set of preperiodic (resp. periodic) points of $F$ contained in $V$ is Zariski dense in $V$ , then $V$ itself is preperiodic (resp. periodic).
\end{conjecture}
It basically says if a variety intersects with the periodic points of $F$ infinitely many times, then the variety is also periodic under $F$. 

The original version of Manin-Mumford conjecutre was proved by Rayanaud (see~\cites{Raynaud1983,Raynaud1983incollection}) and many others (see~\cites{Ullmo1998,Zhang1998,Buium1996,Hrushovski2001}, and \cite{PinkRoessler2004}). Baker and Hsia first proved that the dynamical Manin-Mumford conjecture is true for any line on an affine space of dimension $2$. Then, many counterexamples were given (see \cites{pazuki2010, pazuki2013, GTZ2011}, and \cite{YuanZhang2017}). However, \cite{GTZ2011} also shows the conjecture is hold for some cases, and after that many cases have been proved to be valid (see~\cites{Fakhruddin2014, GNY2018, GNY2019, GhiocaTucker2010}, and \cite{DujardinFavre2017}). Now, we believe the dynamical Manin-Mumford conjecture is only failed for some special families.

The following is my result on the dynamical Manin-Mumford conjecture. 
\begin{theorem}\label{the dynamical Manin-Mumford conjecture}
Let $V$ be a variety of $\Aff_K^N$ and $F$ is a restricted lift of $p$th power, and assume that $\bar{F}^l$ is a finite surjective morphism of $red(V)\coloneqq \bar{V}$ for some integer $l\geq 1$, i.e. $\bar{F}^l:\bar{V} \to \bar{V}$. Then, if $V\cap \Per(F)$ is dense in $V$, then $V\cap\Per(F^l)=V\cap\Per(F)$ is periodic.
\end{theorem}

\subsection*{Mordell-Lang Conjecture}
The idea of Mordell-Lang conjecture is this. Let $G$ be a semiabelian variety over $\C$ and $V$ is a subvariety of $G$. Taking an element $g$ in $G$, then $\{n\in\Z\mid ng\in V\}$ is a union of finite many arithmetic progressions, i.e.
\[
\{n\in\N\mid ng\in V\}=\cup_{i=1}^k\{a_ik+b_i\mid k\in \Z\}
\]
for some integer $a_i$ and $b_i$. A more general statement is proved by Faltings (see~\cite{Faltings}) and Vojta (see~\cite{Vojta}). A dynamical version is inspired by Ghioca and Tucker (see~\cite{GhiocaTucker2009}). Many different versions of the dynamical Mordell-Lang conjecture have also been given (see~\cite{BellGhiocaTucker2010}, \cite{B-G-K-T-2012}, \cite{SilvermanBianca2013} and \cite{Xie2017}). Let me state the dynamical Mordell-Lang conjecture.
\begin{conjecture}
Let $X$ be a qusi-projective variety and $V$ be a subvariety of $X$. Let $f:X\to X$ be an endomorphism. For any $x\in X$, the set $\{n\mid f^n(x)\in V\}$ is a union of finitely many arithmetic progressions. 
\end{conjecture}

If we look closer to the conjecture and compare the original version with the dynamical one, we will find that the above conjecture is not complete. The $n$ in the original version can be positive and negative, so we shouldn't only consider forward orbits of a dynamical system but also backward orbits of a dynamical system. The backward orbit was asked by S.-W. Zhang and solved by Xie for lift of Frobenius. They made the following conjecture.
\begin{conjecture}\label{conjecture:inverse mordel-lang}
Let $X$ be a quasi-pojective variety over a algebraic closed field $\overline{K}$ and $F:X \to X$ be a finite endomorphism. Let $\{b_i\}_{i\geq 0}$ be a sequence of points in $X$ satisfying $F(b_i)=b_{i-1}$ for all $i\geq 1$. Let $V$ be a subvariety of $X$. Then the set ${n\geq 0\mid b_n \in V }$ is a union of finitely many arithmetic progressions.
\end{conjecture}

The condition $F(b_i)=B_{i-1}$ is called coherent condition, and we can actually made a conjecture without such condition, but there are counterexamples. For instance, consider a map $F:=(x^2,x^3):\Aff_\C^2\to\Aff_\C^2$ with $b_0=(1,1)$, and let $V$ be the diagonal of $\Aff_{\C}^2$. There are other preimages of $(1,1)$ by $F^n$ is not on the diagonal, and note that $(1,1)$ is a fixed point of $F$, so I can chose $b_{2^i}$ to be $(1,1)$ and $b_i$ to be any other preimages not on the diagonal. Then, this construction will make $\{i\mid b_i\in V\}$ not a union of infinitely many arithmetic progressions. However, this example is cheating in the sense that actually $(1,1)$ is belonging to $F^{-n}((1,1))$ for all $n$. Indeed, Theorem~\ref{finialthm} shows that if $F$ is eventually stable and $V$ is defined over the defining field of $F$, then the coherent condition is automatically satisfied. A polynomial $f$ is eventually stable, if the number of factors of $f^n$ is bounded. The eventually stable condition on an endomorphism follows the similar sense (see Definition~\ref{eventuallystable} for more details). Eventually stable doesn't exit an easy criteria. Rafe and Alon (see~\cite{RafeAlon2016}) made some conjectures about the eventual stability.

In the end, I prove a theorem similar to Xie's by generalize the endomorphism to a restricted lift of $p$-th power.
\begin{theorem}\label{Inverse Mordelll-Lang conjecture}
Let $\{a_n\}$ be a coherent backward orbit of $a_0$ by a restricted lift of $p$th power $F$. Assume that there exists a subsequence $\{a_{n_i}\}$ contained in an irreducible subveriety $V$ with $\bar{F}^{l}(\bar{V})=\bar{V}$ for some integer $l$. We further assume that $b_0$ is algebraic. Then, we have $\{n_i\}$ is arithmetic progressive.
\end{theorem}

\subsection*{Tate-Volch Conjecture}
Tate-Volch conjecture is a relative new question in algebraic geometry by \cite{TateVoloch1996}.
\begin{conjecture}\label{conjecture: tate-voloch}
Let $A$ be a semiabelian variety over $K_v$, the complete field with respect to a valuation $v$ on $K$, and $V$ a subvariety of $A$. Then there exists $c > 0$ such that for any torsion point $x \in A$, we have either $x\in V$ or $d(x,V)>c$.
\end{conjecture}
The $d(x,V)$ is the ``distance" of the point $x$ to $V$. To avoid any complicate definition, the $d(x,V)$ is the infinite norm of $V(x)$, evaluating $x$ by plugging to the defining equations of $V$.
\cite{Scanlon1999} showed the conjecture is true over any finite extension of $\Q_p$, and \cite{Buium1996onTateVoloch} proved the conjecture is true for periodic points of any lift of Frobenius endomorphism $f:V\to V$. I prove
\begin{theorem}\label{dynamical tate-voluch conjecture}
Let $F$ be a restricted lift of $p$th power over $\mathbb{\C}_p$. Then, for any variety $V$, there exists an $\varepsilon>0$ such that, for all $x\in \Per(F)$, either $x\in V$ or $d(x,V)>\varepsilon$.
\end{theorem}

The idea of using the theory of perfectoid space is straight forward since the construction of a perfectoid space is literately a dynamical system. Let $F:X\to X$ be a restricted of $p$-th power over $\C_p$. A correspondent perfectoid space $X^{perf}$ can be created with an endomorphism $F^{perf}$ such that the following diagram commutes
\[
\begin{tikzcd}
X\arrow[r,hook]\arrow[d,"F"] & {X}^{perf}\arrow[d,"{F}^{perf}"]\\
X\arrow[r,hook] & {X}^{perf}
\end{tikzcd}
\]

The tilt of the above diagram is
\[
\begin{tikzcd}
{X}^\flat\arrow[r,hook]\arrow[d,"{F}^\flat"] & {X}^{perf,\flat}\arrow[d,"{F}^{perf,\flat}"]\\
{X}^\flat\arrow[r,hook] & {X}^{perf,\flat}
\end{tikzcd}
\]
Using the theory of perfectoid spaces, one has a topological homeomorphism $\rho$ with an important approximation lemma to pass $X^{perf}$ to $X^{perf,\flat}$ with $\rho\circ F^{perf}=F^{perf,\flat}\circ \rho$. This connect the right edges of the above two diagram. On the other side, $X$ can be reduced to $\overline{X}$, defined over the residue field of $\C_p$ with $\overline{F}\e F\mod \C_p^{\circ\circ}$, where $\C_p^{\circ\circ}$ is the maximal ideal of the valuation ring of $\C_p$. The dynamical system $(\overline{X},\overline{F})$ can be naturally embedded into $(X^\flat,F^\flat)$ since the residue field of $\C_p^{\flat}=\widehat{\overline{\F_p(t^{1/p^{\infty}})}}$ is just a subfield of $\C_p^{\flat}$. In conclusion, we have reduction and embedding on one side and topological homeomorphism on the other side to connect the two commutative diagrams. Using this framework, we can pass results that is true over a field of characteristic $p$ to its correspondent field of characteristic $0$.

Finally, let me outline this chapter. I will introduce more about the theory of perfectoid spaces in the next sections, and carefully construct our framework of perfectoid spaces carefully in the third section. The fourth section will use the framework in the second section to study the periodic points of $F$. Then, I give proof of all theorems and required lemmas in the last section.
\section{An elementary introduction to perfectoid spaces}
What is a perfectoid space? Rather than giving a definition directly, I am going to introduce what the rigid geometry analysis is and where the perfectoid spaces are on the big picture. We start from an example from rigid analysis spaces which has been studied and applied to many questions in the number theory, Berkovich spaces. Let $K$ be an algebraic closed and completed field with respect to a non-archimedean value. We know $K$ is Hausdorff, but is totally disconnect and not locally compact. These topological properties make it difficult to use analysis to study $K$. Tate dealt with this problem and developed a method, now called rigid geometry analysis. 

The general problem is that given a geometric object $V$ over $K$, the topology on the object may be not suitable for the question that we want to study. The idea by Tate is constructing a larger space $X$ which has "good topological properties", and then embedding the geometric object $V$ into $X$ such that the underlying topology $V$ is unchanged. These preferred properties depend on the context of discussion. For example, Berkovich spaces are spaces that can properly and naturally define Laplacian operation analogue to the classical one on $\C$. \cite{Baker2008} gives a good explanation for the motivation of discussion Berkovich spaces. Perfectoid spaces serve different purposes. At first, Fontaine and Wintenberger found a canonical isomorphism of Galois groups\cite{FontaineWintenberger1979}, and the special case is the following theorem.
\begin{theorem}
The absolute Galois groups of $\Q_p(p^{1/p^{\infty}})$ and $\F_p((t^{1/p^{\infty}}))$ are canonically isomorphism.
\end{theorem}
This theorem is a convincing evidence pointing to that the worlds of characteristic $0$ and characteristic $p>0$ have some connection after adjoining all $p$-th power roots of $p$. One can formulate a more ambitious conjecture saying that a category of some geometry objects over a field of characteristic $0$ and characteristic $p$ is equivalent. Perfectoid spaces are the perfect answer to this wild speculation. Of course, the guess does not show himself immediately after Fontaine and Wintenberger's theorem, and we refer readers to \cite{Scholze2012}, the first paper of Scholze on the perfectoid spaces about the details and the intuition to formulate the conjecture. The theory of perfectoid space has many applications (see \cites{scholze2014,scholze2013}). Especially, it gives a complete solution on the local {L}anglands problem (see~\cite{scholze2013llp}) Also, \cite{Bhatt2014} gave a friendly introduction to the perfectoid spaces. In the following paragraphs, we should focus on how to define a perfectoid field and a perfectoid space, and we will compare the perfectoid spaces with the Berkovich space to let readers feel connection between new things and a canonical stuff.
\subsection{Valuations}
First, we recall the definition of valuations which is not the valuation that we are used to use in the number theory.
\begin{definition}
Let $R$ be a commutative ring. A valuation on $R$ is a multiplicative map $x:R\to \Gamma$, where $\Gamma$ is a totally ordered abelian group, written multiplicatively, such that $x(0)=0$, $x(1)=1$ and $x(a+b)\leq\max\{x(a),x(b)\}$ for all $a,b\in R$.

Given $R$ a topological structure, a valuation $x$ of a ring $R$ is continuous if the subset $\{r\in R\mid x(r)<\gamma\}\subset R$ is open for all $\gamma\in \Gamma$.
\end{definition}
\begin{remark}
The major difference of this definition with the one used by most number theorists is that the codomain $\Gamma$ doesn't need to be $\R$. It is tragic that we did not have a universal name for these similar terminologies. 

I will use $x$ to denote the valuations I just defined, use $|\cdot|$ to denote the $p$-adic or $t$-adic value of a field or a ring, and use $\|\cdot\|$ to denote the norm on a space over $K$. More precisely, the field $K$ in this chapter has a residue field of characteristic $p>0$, and the value of $K$ satisfies $|p|<1$. A ring $R$, which will be defined next, will equip with the Gauss norm, also denoted by $|\cdot|$, induced by the norm on $K$. Remember the Gauss norm is one of many valuations that can be defined on $R$.  
\end{remark}

\begin{definition}
An $N$-variables Tate algebra $R$ over a field $K$ is
\[
R=K\la X_1,\ldots, X_N\ra=\left\{\sum a_IX^I\mid |a_I|\to 0\text{ as }\|I\|\to \infty\right\}
\]
where, for a multi-index $I=(i_1,\ldots,i_N)\in\N^N$, $X^I\coloneqq  X_1^{i_1}\cdots X_N^{i_N}$ and $\|I\|\coloneqq \sum_{j=1}^N i_j$. A Tate ring is a subring of the ring of formal power series, and is also naturally a topological ring where the metrics is induced by the Gauss norm $|\cdot|$ where, for an element $f\in R$,
\[
|f|\coloneqq\max_I\{a_I\}.
\]

I further define $R^{\circ}\coloneqq K^{\circ}\la X_1,\ldots,X_N\ra$ where $K^\circ$ is the valuation ring of $K$.
\end{definition}

The abstract definition of a Tate-algebra $R$ is a topological $K$-algebra that has a subring $R_0\subset R$ such that $\{aR_0\mid a\in K^{\times}\}$ forms a basis of open neighborhood of $0$. We then say that an element $f\in R$ is bounded if $\{f^n\}\subseteq aR_0$ for some $a\in K$. Finally, we define $R^{\circ}$ be the ring of all power-bounded elements. I do not need such generality, but the above definition satisfies the general one. 

We are going to give two valuations on Tate $K$-algebra that are not equivalent to the valuations used in the number theory. The first example is not a continuous valuation, and the second one is.
\begin{example}
Let $R=K\la X,Y\ra$. For $f\in R\setminus\{0\}$, I can write $f$ as the follow
\[
f=\sum_{i=n_0}^\infty X^i\sum_{j=m(i)}^\infty a_{i,j}Y^j
\]
where $a_{i,m(i)}$ is not zero for all $i$. Consider the set $\Gamma=\gamma^\Z\times \zeta^\Z$ with the lexicographic order where $\gamma$ and $\zeta$ are just two symbols. In order ward, I say $(\gamma^i,\zeta^j)<(\gamma^{i'},\zeta^{j'})$ if $i<i'$ or ($i=i'$ and $j<j'$). A multiplication map $x_1: R\to \Gamma\cup\{0\}$ can be defined as
\[
x_1(f)=\gamma^{-n_0}\times \zeta^{-m(n_0)},
\]
for $f\neq 0$, and let $x_1(0)=0$. Readers can easily check that this is a valuation of $R$. The $x_1(0)$ can be thought as $(\gamma^{-\infty},\zeta^{-\infty})$, and it shows that, for an $f$ that is closed to $0$ in $R$ with respect to the Gauss norm, the $x_1(f)$ can be very far away from $(\gamma^{-\infty},\zeta^{-\infty})$. This indicates that $x_1$ is not continuous.
\end{example}
\begin{example}
Now, we give a more interesting example which can also be found in Scholze's first paper on the perfectoid spaces and many other lecture notes. We will focus on proving that the valuation we are going to construct is continuous. 

Let $\Gamma=\R_{>0}\times \gamma^{\Z}$ be a multiplicative group with the lexicographic order, so we have $(r,\gamma^{-1})<(1,\gamma^{-1})<(1,1)$ for $r<1$. Informally, we can consider $\gamma$ is infinitely closed to $1$. Given $f\in \C_p\la T\ra$, and expressing $f=\sum_{i}a_iT^i$, we can define a rank-$2$ valuation $x_2$ on the $\C_p\la T\ra$ by
\[
x_2(f)\coloneqq\max_i\{(|a_i|,\gamma^i)\}=(|f|_{D(0,1)},\gamma^{i_0})
\]
with $i_0=\max\{i\mid |a_i|=|f|_{D(0,1)}\}$ where $|\cdot|_{D(0,1)}$ is the Gauss norm. Somehow, we can understand this point as the starting point of the branch of the Gauss point down to $0$, or we can think this is a valuation correspondent to the open disc $D^-(0,1) \coloneqq \{a\in \C_p\mid |a|<1\}$. All higher rank valuations over $\C_p\la T\ra$ are constructed similarly like our example, and they are classified as the type $5$ points of the Berkovich spaces. 

To show that this valuation is continuous. we first note that $|f-g|_{D(0,1)}<|f|_{D(0,1)}$ implies $|f|_{D(0,1)}=|g|_{D(0,1)}$. Therefore, if we fix an element $f\in \{f\in\C_p\la T\ra\mid x_2(f)<(r,\gamma^n)\}$, denote this set by $S$, the open ball $\mathcal{B}(f,r)\coloneqq\{g\in\C_p\la T\ra\mid |g-f|_{D(0,1)}\leq r\}$ with $r<|f|_{D(0,1)}$ is contained in $S$. It implies that $S$ is open, and so the valuation $x_2$ is continuous.
\end{example}

We will not care too much about what the valuations look like, but there are some valuations we are interested. Let $x\in \C_p$ and $f(T)\in\C_p\la T\ra$. The valuation given by evaluating $f$ at the point $x$, $|f(x)|$, is known as the type $1$ point on the Berkovich space. These and some other slightly generalized points are what we would like to focus on. They are called \textbf{rigid points} and will be defined in the following section.

\subsection{Huber's adic space}
Scholze chose to work on Huber's adic space. Huber's adic space over a topological ring $R$ as a set is a set of valuations of $R$. Instead of naively putting all valuations of $R$ in the set like the Berkovich space considering all rank-$1$ valuations, Huber's adic space over $R$ places a restriction on the choice of valuations which makes it possible to given a scheme structure on the space.

\begin{definition}
An affinoid $K$-algebra is a pair $(R,R^+)$, where $R$ is a Tate $K-$algebra and $R^+$ is an open and integrally closed subring of $R^\circ$.
\end{definition}
Then, Huber associated a space of equivalent continuous valuations to a pair $(R,R^+)$, and we call the space Huber's adic space or adic space for short. The definition of the equivalent relations of valuations can be found in P. Scholze's paper, and the only property that matters in this paper is that if $x$ and $x'$ are  equivalent, then $\text{Supp}(x)=\text{Supp}(x')$ where $Supp(x)\coloneqq\{f\in R\mid |f|_x\neq 0\}$. To simplify our notation, we denote $x(f)$ by $|f(x)|$. 
\begin{definition}
For a pair of affinoid $K$-algebra $(R,R^+)$, Huber's adic space $X$ is 
\[
\Spa(R,R^+)\coloneqq\{x\text{ continuous valuation on }R\mid f\in R^+:|f(x)|\leq 1\}/\sim,
\]
where $\sim$ is the equivalent relation.

A Huber's adic space $X$ equips with the topological structure which the rational subsets,
\[
U\left(\dfrac{f_1,,\ldots,f_n}{g}\right)\coloneqq\{x\in\Spa(R,R^+)\mid |f_i(x)|\leq|g(x)|\quad\forall i\},
\]
form the base of open sets.
\end{definition}

We should note that even thought the Berkovich spaces and the adic spaces $\Spa(\C_p\la T\ra,\C_p^\circ\la T\ra )$ are not different much as sets, their topological structures are completely different. The major difference is that the rational sets form a scheme structure on an adic space, and the Berkovich spaces do not have a scheme structure.
We will not use the scheme structure in this paper, and we refer readers to Scholze's original paper for more details. 

As a subject of the rigid geometry, we should embed interesting points into these spaces. Let us talk about rigid points on Huber's adic spaces.
\begin{definition}
Given an affinoid $K$-algebra $(R,R^+)$, we call a point $x\in\Spa(R,R^+)$ a \textbf{rigid point} if $\{f\in R\mid |f(x)|=0\}$ is a maximal ideal of $R$, and we define
\[
\mathrm{R}(\Spa(R,R^+))\coloneqq\{x\in\Spa(R,R^+)\mid x\text{ is a rigid point}\}.
\]
\end{definition}
We will also use $\mathrm{R}(X)$ to denote the closed point of $X$ for a ringed space $X$. Let $R$ be an $N$-variable Tate $K$-algebra, we note that a closed point $x\in\Aff_{K^{\circ}}^N$ defines a rigid point on $\Spa(R,R^\circ)$. Moreover, if $K$ is algebraically closed, then there is a bijection from the points on $\Aff^N_{K^\circ}$ to the rigid points on $\Spa(R,R^\circ)$.

\subsection{Perfectoid fields, perfectoid $K$-algebra, affinoid perfectoid spaces and perfectoid spaces}
An affine piece of a perfectoid space is an adic space over an "almost perfect" $K$-algebra. Let us begin this section by defining perfectoid fields. 
\begin{definition}
I say $K$ is a \textbf{perfectoid field} if $|K|\subseteq \R_{\geq 0}$ is dense in $\R_{\geq 0}$ and the Frobenius map $\Phi:K^\circ/p\to \Phi:K^\circ/p$ is surjective. 
\end{definition}
Why are this field perfect? Scholze gave the following construction. Given a perfectoid field $K$, he associate $K$ with another field of characteristic $p$, denoted by $K^\flat$, and proved that the category of finite extensions of $K$ and the category of finite extensions of $K^\flat$ are equivalent. Actually, Fontaine and Wintenberger's theorem is a special case of this theorem. 

Next, let us define the perfectoid $K$-algebra. Let us fix an element $\varpi\in K^\circ$ satisfying $|p|\leq |\varpi|<1$. It does not really matter which $\varpi$ I fix, so I can choose randomly or just let $\varpi= p$. 
\begin{definition}
A perfectoid $K$-algebra is a Banach $K$-algebra $R$ which has the subset $R^\circ\subset R$ of powerbounded elements open and bounded and the Frobenius morphism $\Phi:R^\circ/\varpi\to R^\circ/\varpi$ surjective. Morphisms between perfectoid $K$-algebras are the continuous morphisms of $K$-algebras.
\end{definition}
\begin{remark}
The condition that $\Phi:R^\circ/\varpi\to R^\circ/\varpi$ is surjective is equivalent to the condition that $\Phi:R^\circ/\varpi^{1/p}\to R^\circ/\varpi$ is injective. The later articles by Scholze and other people use the later condition as the definition.
\end{remark}
The field of coefficients $K$ is naturally contained in $R$, and $R^\circ \cap K$ is $K^\circ$. Since the Frobenius morphism is surjective on $R^\circ/\varpi$, It implies that $K$ is a perfectoid field. Hence, a perfectoid $K$-algebra can be simply through as a quotient Tate $K$-algebra for a perfectoid field $K$ with the preimage of every transcendental element by the Frobenius map not empty. 
\begin{example}\label{example:perfectoid ring}
One can easily check that $\Q_p(p^{1/p^{(\infty)}})\coloneqq \cup_{i}\Q_p(p^{1/p^i})$ and $\F_p(t^{1/p^{(\infty)}})$ are perfectoid fields, and their algebraic closed and completed field, $\C_p$ and $\widehat{\overline{F_p(t)}}$, are also perfectoid. In this paper, I only use the latest two fields. 

Let $K$ be a perfectoid field, and I can give a perfectoid $K$-algebra like the following
\[
K\la T^{(\infty)}\ra=K\la T^{(0)},T^{(1)},T^{(2)},\ldots\ra/(T^{(i)}-(T^{(i+1)})^p\mid i\geq 0).
\]
More general perfectoid $K$-algebras could be like $K\la\ T_1^{(\infty)}, \ldots, T_N^{(\infty)}\ra$ or like $K\la\ T_1^{(\infty)}, \ldots, T_N^{(\infty)}\ra/I$ where $I$ is an ideal of the ring $K\la\ T_1^{(\infty)}, \ldots, T_N^{(\infty)}\ra$.  
\end{example}
We finally have enough terminologies to define perfectoid spaces.
\begin{definition}
A perfectoid affinoid $K$-algebra $(R,R^+)$ is an affinoid $K$-algebra over a perfectoid $K$-algebra. The \textbf{affinoid perfectoid} space, an affine piece, is an adic  space $X$ over a perfectoid affinoid $K$-algebra, i.e. $X=\Spa(R,R^+)$ is an affinoid perfectoid space if $R$ is a perfectoid $K$-algebra. Then, one defines \textbf{perfectoid spaces} over $K$ to be the objects obtained by gluing affinoid perfectoid spaces. The topology on perfectoid spaces follows from the topological structure on Huber's adic space.
\end{definition}
\begin{example}
One can define a \textbf{projective perfectoid space} by gluing affinoid perfectoid spaces. Let $R_i^{perf}=K\la T_{i,0}^{(\infty)},\ldots T_{i,i-1}^{(\infty)},T_{i,i+1}^{(\infty)},\ldots T_{i,N}^{(\infty)}\ra$ for $i=1,\ldots,N$, and we define the $i$th piece affinoid perfectoid space to be
\[
\Aff_i^{N,perf}=\Spa(R_i, R_i^\circ).
\]
For any $i\neq j$, we glue 
$$U_{i,j}:=U(1,T_{i,0}^{(0)},\ldots,T_{i,i-1}^{(0)},T_{i,i+1}^{(0)},\ldots,T_{i,N}^{(0)};T_{i,j}^{(0)})\subset \Aff_i^{N,perf}$$
and
$$U_{j,i}:=U(1,T_{j,0}^{(0)},\ldots,T_{j,j-1}^{(0)},T_{j,j+1}^{(0)},\ldots,T_{j,N}^{(0)};T_{j,i}^{(0)})\subset \Aff_j^{N,perf}$$
by the transition function $\phi_{i,j}:U_{i,j}\to U_{j,i}$ induced by the pullback function  
\[
\phi_{i,j}^{*}(T_{j,k}^{(n)})=\dfrac{T_{i,k}^{(n)}}{T_{i,j}^{(n)}}\quad\text{and}\quad \phi_{i,j}^*(T_{j,i}^{(n)})=\dfrac{1}{T_{i,j}^{(n)}}
\]for all $n$.
\end{example}
\subsection{Visualize a one dimensional perfectoid space}
I would like to give some graphs to image a one dimensional perfectoid space, and also give a preliminary construction for a framework of applying perfectoid spaces on dynamical systems. To illustrate the idea, I will only give a proof if it is necessary. The Example~\ref{example:perfectoid ring} will be used. Let $K=\C_p$ associate with $K^\flat=\widehat{\overline{\F_p(t)}}$. Let $R_\bullet=K\la T^{\bullet}\ra$ and $R^\flat_{\bullet}=K^\flat\la T^{\bullet}\ra$ for $\bullet=0,1,2\ldots$ and $\infty$. I let $X_{\bullet}=\Spa(R_{\bullet},R_{\bullet}^{\circ})$ and similarly define $X_\bullet^\flat$. The $\Aff_{K}^1$ can be embedded into $X_0$, and the embedding is a bijection on the closed points of $\Aff_K^1$ and the rigid points of $X_0$ (see Figure \ref{perfectoidfig1}).

\begin{figure}
    \centering
    \includegraphics[scale=0.3]{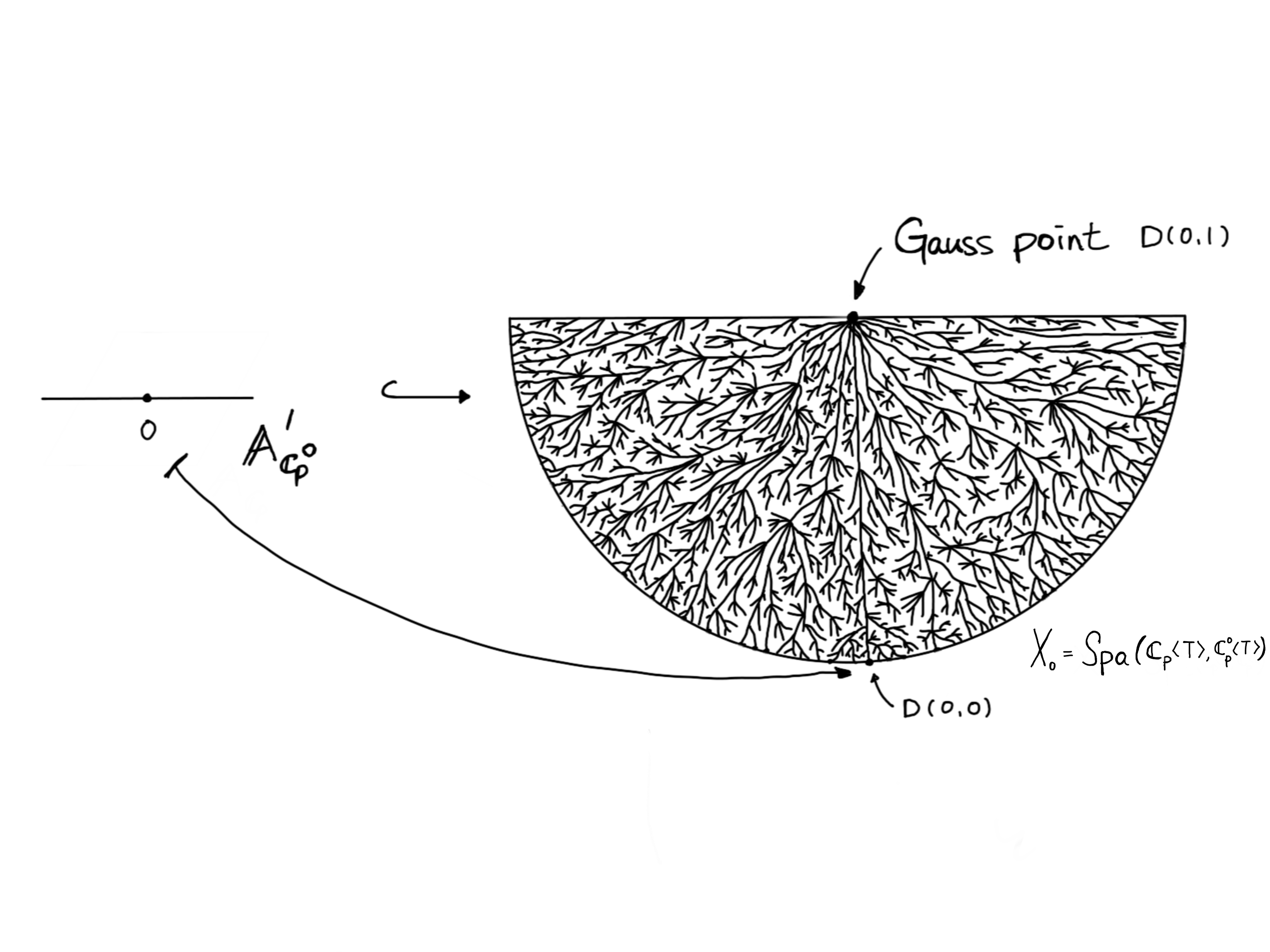}
    \caption{Adic space over $\C_p$}
    \label{perfectoidfig1}
\end{figure}
Note that each $X_i$ is just a Huber's adic space over the pair $(K\la T\ra,K^\circ\la T\ra)$. Scholze has already given an image of such space in his first paper on Perfectoid theory, so I refer readers to his paper for more details. According to Scholze's description, we can image $X_i$ as a half of a Berkovich tree with infinite many point accumulating around each type II point. For each branch point except the Gauss point on a Berkovich tree, the set of extra points is $\mathbb{P}^1_k$ where $k$ is the residue field of $K$, and one for the Gauss point is $\Aff_k^1$ (see Figure \ref{perfectoidfig2}). It is not precisely analogue, but one can think those extra points is a kind of blow-up at each branch point since when we consider the blow-up on a variety, it gives an affine or projective line over the field of definition which is exactly the residue field of the coordinate ring.

\begin{figure}
    \centering
    \includegraphics[scale=0.3]{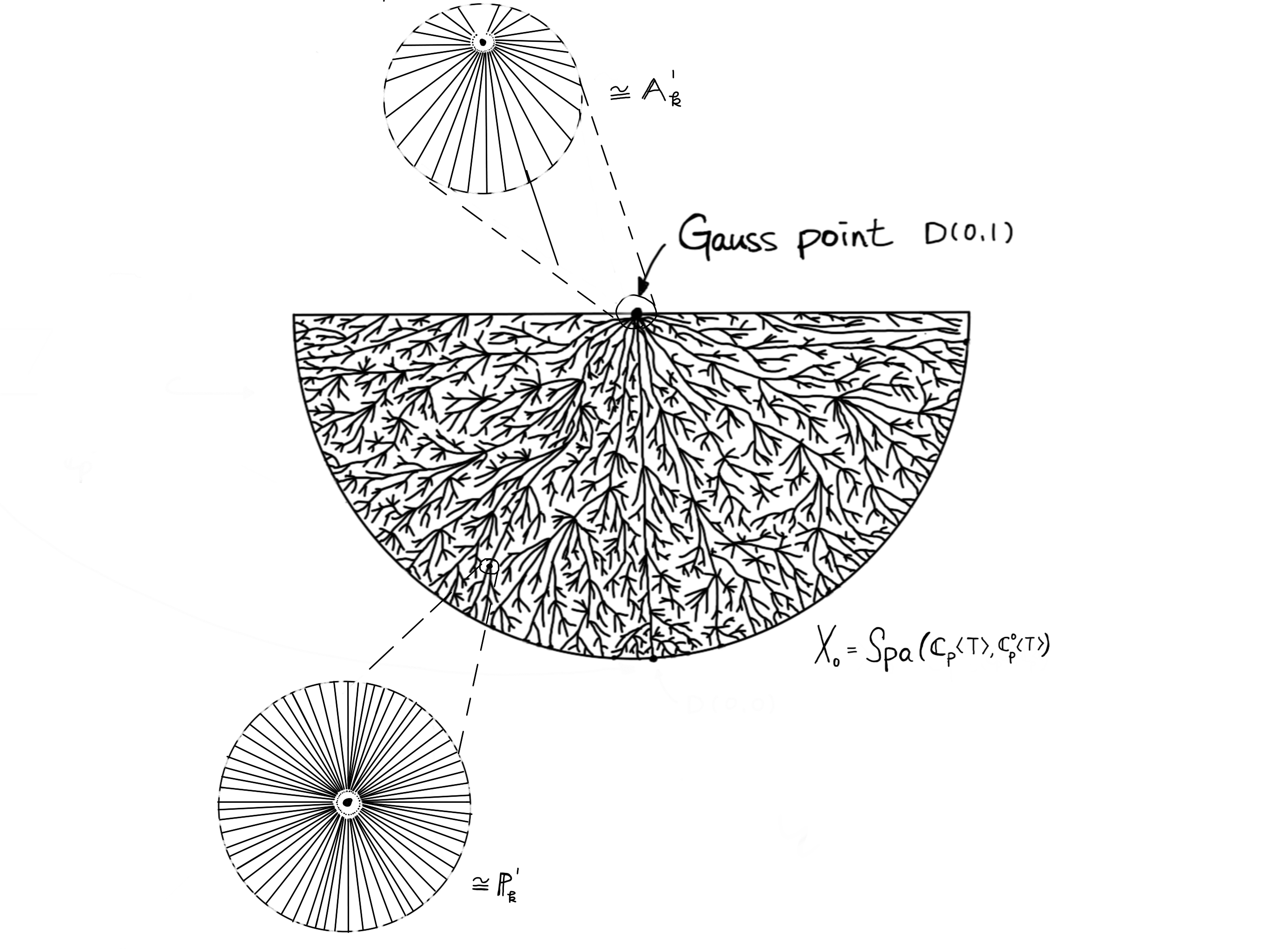}
    \caption{Type V points of Adic space}
    \label{perfectoidfig2}
\end{figure}

Now, I will focus on $X_0$ and $X_1$ for now. Since $R_0$ contains in $R_1$, a valuation $x_1\in X_1$ naturally maps to a valuation on $X_0$ by taking restriction. Note that the relation $T^{(0)}=(T^{(1)})^p$ implies that the induced map on $X_2$ to $X_2$ is $p$ to $1$. For example, given a $1$-dimensional valuation with respect to a disc $D(a,r)$ in $\C_p$ with $a\neq 0$, the valuation of $T_0$ is
\[
|T^{(0)}|_{D(a,r)}=|(T^{(1)})|^p_{D(a,r)}=|(\dfrac{1}{r}(T^{(1)}-a))^p|_{D(0,1)}=|\dfrac{1}{r^p}(T^{(1)})^p-a^p|_{D(0,1)}.
\]
Hence all $p$-th roots of $a^p$ will give the same valuation on $T^{(0)}$, so I can think that $X_{i+1}$ to $X_i$ is a $p$ to $1$ covering map, and the inverse limit of these covering map is a perfectoid space (see the top row of Figure \ref{perfectoidfig3}).

\begin{figure}
    \centering
    \includegraphics[scale=0.35]{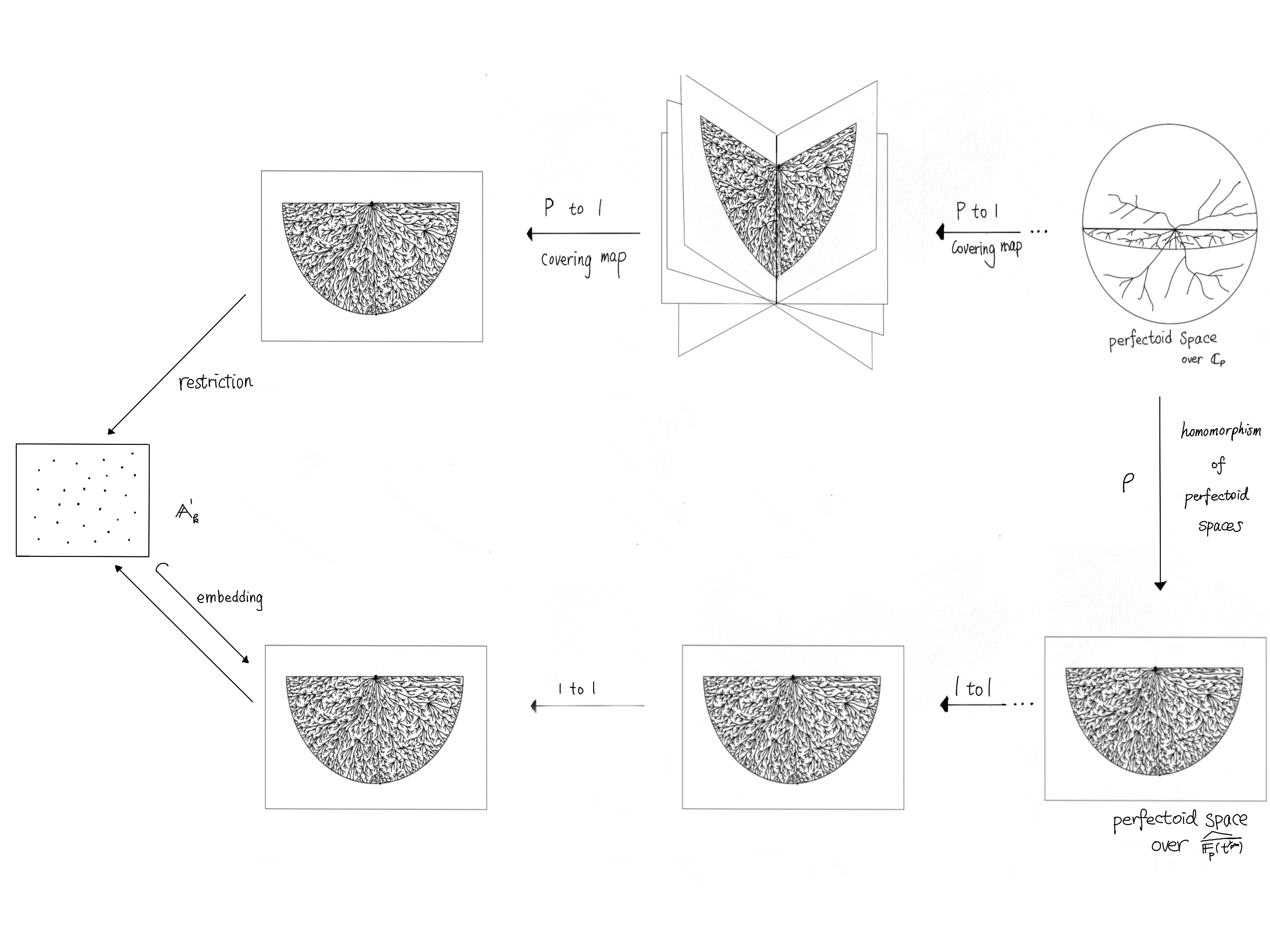}
    \caption{Perfectoid Spaces}
    \label{perfectoidfig3}
\end{figure}

For spaces over the field $K^\flat$, the only difference is that the map from $X_{i+1}^\flat$ to $X_i^\flat$ is a one-to-one map since Frobenius map is bijection over a field of positive characteristic (see the button row of Figure \ref{perfectoidfig3}). Now, there are two maps to connect the top row and the bottom of Figure \ref{perfectoidfig3}. The left side is elementary. In $K$, the valuation ring $K^\circ$ quotient out by the maximal ideal is $k$, and it can be naturally embedded to $K^\flat$. On the other side, Scholze's theory says there is a homomorphism from the perfectoid space over $K$ to one over $K^\flat$. In the next section, I will formally prove the idea given above. 
\section{Construct our framework}~\label{section:3}
\subsection{Passing to the perfectoid spaces from the affine spaces}
Let me first give the following setups. We let $K=\C_p$ and $K^{\flat}=\widehat{\overline{\F_p(t)}}$, and let $R=K[X_1,\ldots,X_N]$ and $R^\flat=K^\flat[X_1,\ldots,X_N]$ be the coordinate rings of $\Aff_K^N$ and $\Aff_{K^\flat}^N$ respectively. The normalized value of a completed field $L$ is denoted by $|\cdot|:L\to\R$, and we denote the valuation ring $\{x\in L\mid |x|\leq 1\}$ of $L$ by $L^{\circ}$ and the maximal ideal $\{x\in L^\circ\mid |x|<1\}$ by $L^{\circ\circ}$. Finally, the residue field $k$ of $K$ is $k\coloneqq  K^\circ/K^{\circ\circ}$ and note that $K^{\flat,\circ}/K^{\flat,\circ\circ}$ is also isomorphic to $k$ which can be embedded into $K^\flat$ naturally.

The $R^\circ$ and $R^{\flat,\circ}$ are the sets of power-bounded elements in the correspondent rings, and it is easy to check that they are $R^\circ=K^\circ[X_1,\ldots,X_N]$ and $R^{\flat,\circ}=K^{\flat,\circ}[X_1,\ldots,X_N]$ since $|f^n|=|f|^n$ where $|\cdot|$ is the Gauss norm on $R$ and $R^\flat$.

We call a function $F:\Aff_K^N\to\Aff_K^N$ the \textbf{lift of $p$th power}, written $F=(F_1,\ldots,F_N)$ for $F_i\in R^\circ$, if we can find an element $\varpi$ with $|p|\leq |\varpi|<1$ such that
\[
F_i\e(G_i)^p\mod\varpi
\]
where $G_i$ is in $R^\circ$ for all $i$. We will call a function $F:\Aff^N_K\to\Aff^N_K$ \textbf{restricted lift of $p$th power} if $F$ is a lift of $p$th power which satisfies 
\[
F:\Aff_K^N\setminus\Aff_{K^\circ}^N\to \Aff_K^N\setminus\Aff_{K^\circ}^N\qquad\text{and}\qquad\|F(X)\|>\|X\|
\]
where $\|X\|=\max_i\{|X_i|\}$ if we write $X=(X_1,\ldots,X_N)$. The following lemma shows that this kind of maps exists and is not a lift of Frobenius. 
\begin{lemma}\label{keylemma}
Let $K=\C_p$ and $F=(F_1,\ldots,F_N):\Aff^N_K\to\Aff^N_K$ be a lift of $p$th power, $\varpi\in K^\circ$ with $|p|\leq|\varpi|<1$, and assume $F_i$ can be expressed as
\[
F_i=(G_i)^p+\varpi f_i\qquad\text{for each }i
\]
where $f_i\in R^\circ$ has degree smaller that the degree of $G_i^p$ and where $G_i$ is of the form
\[
G_i(X_1,X_2,\ldots,X_N)=\sum_{j=1}^{m_i}c_{i,j}X_i^{q_{i,j}}
\]
where $c_{i,j}\in K^{\circ}\setminus K^{\circ\circ}$, $m_i\in \N$ and $q_{i,j}$ are a power of $p$ with $q_{i,j}>q_{i,j'}$ if $j>j'$.  Then, $F$ is a restricted lift of $p$th power.
\end{lemma}
\begin{proof}
Given an element $\mathrm{x}\coloneqq  (x_1,\ldots,x_N)\in\Aff_K^N\setminus\Aff_{K^\circ}^N$, we can find an integer $n$ such that $|x_n|=\max_i\{|x_i|\}\geq |x_i|$ for all $i$ and $|x_n|>1$. It follows that
\[
|F_n(\mathrm{x})|=\left|\sum_{j=1}^{m_n} c_{n,j}x_n^{q_{n,j}}+\varpi f_n(\mathrm{x})\right|=|c_{n,m_n}x_n^{q_{n,m_n}}|>|x_n|>1,
\]
so $F(\mathrm{x})$ is in $\Aff_K^N\setminus\Aff_{K^{\circ}}^N$ and $\max_i\{F_i(\mathrm{x})\}>\max_i\{x_i\}$.

\end{proof}

\subsubsection{Adic space}
From now on, we just fix a restricted lift of $p$th power $F:\Aff^N_K\to\Aff^N_K$ and the correspondent element $\varpi\in K^\circ$ in our discussion. 

We let $R^{ad}\coloneqq  K\la X_1,\ldots,X_N\ra$ and $R^{\flat,ad}\coloneqq  K^\flat\la X_1,\ldots,X_N\ra$. The valuation on $R^{ad}$ (resp. $R^{\flat,ad}$) is uniquely extended from the Gauss norm on $R$, so the set $R^{ad,\circ}$(resp. $R^{\flat,ad,\circ}$), power-bounded elements of $R^{ad}$(resp. $R^{\flat,ad}$), is $K^\circ\la X_1,\ldots,X_N\ra$(resp. $K^{\flat,\circ}\la X_1\ldots,X_N\ra$)[BGR, p. 193]. Although most statements in this section will be given for the field $K$, most of them are also true for fields of positive characteristic. We will remind the readers if a statement is only true for fields of characteristic $0$. 

Then, We define  
\[
\Aff^{N,ad}_K\coloneqq Spa(R^{ad},R^{ad,\circ})
\]
and similarly $\Aff^{N,ad}_{K^\flat}\coloneqq  Spa(R^{\flat,ad},R^{\flat,ad,\circ})$.  

The function $F^{ad}:\Aff^{N,ad}_K\to\Aff^{N,ad}_K$ is induced by the pull-back, for each $i$,
\[
F^*(X_i)=F_i(X_1,\ldots,X_N).
\]
More precisely, let $x\in \Aff^{N,ad}_K$. A valuation $F^{ad}(x)$ is given by
\[
|f(F^{ad}(x))|\coloneqq  |F^*(f)(x)|=|f(F_1(X),\ldots,F_N(X))(x)|
\]
where $f$ is belong to $R^{ad}$. 

We identify every valuation $x$ of $R^{ad}$ not in $\Aff^{N,ad}_K$ as the infinite point, denoted by $\infty$, and we have $F^{ad}(\infty)=\infty$ and ${F^{ad}
}^{-1}(\infty)=\{\infty\}$ because $F^{ad}$ is restricted.
Finally, we have the following lemma.
\begin{lemma}\label{tauFFtau}
There exists a natural embedding $\tau:\mathrm{R}(\Aff^N_K)\hookrightarrow\Aff^{N,ad}_K\cup\{\infty\}$ with $\tau(\mathrm{R}(\Aff^N_{K^\circ}))=\mathrm{R}(\Aff^{N,ad}_K)$ and $\tau(\mathrm{R}(\Aff^N_K)\setminus \mathrm{R}(\Aff^N_{K^\circ}))=\{\infty\}$. Moreover, we have a commute diagram
\begin{center}
\begin{tikzcd}
\Aff_K^N\arrow{d}{F}\arrow{r}{\tau}&\Aff_K^{N,ad}\arrow{d}{F^{ad}}\cup\{\infty\}\\
\Aff_k^N\arrow{r}{\tau}&\Aff^{N,ad}_K\cup\{\infty\}
\end{tikzcd}
\end{center}
i.e. $\tau\circ F^{ad}=F\circ\tau$.
\end{lemma}
\begin{proof}
Since we assume $K$ is algebraical closed, a closed point $x$ of $\Aff_K^N$ defines a valuation on $R$ by evaluating the polynomial for the given $x$, and it can be naturally extended to $R^{ad}$, so $x$ maps to the extended valuation over $\tau$. Given $f\in R^\circ$ and a valuation on $R$ that correspondent to a closed point in $\Aff_{K^\circ}^N$, the valuation of $f$ is smaller than $1$ which implies that the extension of the valuation in $R^{ad}$ is in $\Aff^{N,ad}_{K}$. Finally, we will just define $\tau(R(\Aff^N_K)\setminus R(\Aff^N_{K^\circ}))=\{\infty\}$, and finish the construction of the map $\tau$.

The map $\tau$ is trivially an embedding, and we should prove $\tau(R(\Aff^N_{K^\circ}))=R(\Aff^{N,ad}_K)$. It follows by [BGR 2, 6.1.2 Corollary 3], which says that a rigid point of $\Aff^{N,ad}_K$ is correspondent to a maximal ideal $\mathfrak{q}\subset R^{ad}$ and $\mathfrak{q}\cap R$ which is still a maximal ideal of $R$. Finally, the commute diagram follows directly by how we construct $F$, $F^{ad}$ and $\tau$. 
\end{proof}
\begin{remark}
Lemma~\ref{tauFFtau} is still true when $K$ is not algebraic completed, and a closed point $\A$ in $\Aff^N_K=\text{Spec}(R)$ would not correspondent to a point like $(x_1,\ldots,x_N)$. For example, $\A=(X^2+1,Y-1)$, an ideal generated by the polynomials $X^2+1$ and $Y-1$, is a closed point of $\text{Spec}(\Q_3[X,Y])$ which is not correspondent to a point like $(x,y)\in \Q_3^2$. However, We can associate $\A$ with a Galois orbit of a point $\A'\in L^N$ where $L$ is the splitting fields of $\A$ over $K$, and the valuation $\tau(\A)$ is given by, for $f\in R^{ad}$, $|f(\A_i)|$ where $|\cdot|$ is the unique extended valuation on the splitting field over $K$. Note that the choose of $\A_i$ does not affect the valuation. Like the above example, the Galois orbit that represents $\A=(X^2+1,Y-1)$ is $\{(i,1),(-i,1)\}$, and so the valuation $\tau(\A)$ would be, for $f\in R^{ad}$, $|f(i,1)|=|f(-i,1)|$ where the valuation $|\cdot|$ is the one of the $\Q_3(i)$ that uniquely extended from the valuation on the $\Q_3$.
\end{remark}
\subsubsection{Inverse limit space and perfectoid space}
We define an inverse limit space as
\[
\varprojlim_{F^{ad}}\Aff^{N,ad}_K\cup\{\infty\}=\{(x_0,x_1,x_2,\ldots)\mid F^{ad}(x_{i+1})=x_i\in\Aff_K^{N,ad}\cup\{\infty\}\quad\forall i>0\}.
\]
We consider the inverse limit space as a subspace of an infinite product space equipped with the product topology, and we will prove that the inverse limit space is homeomorphic to a perfectoid space later.

Let us focus on this inverse limit space. Let $T:\inverselimit\to\inverselimit$ be
\[
T((x_0,x_1,x_2,\ldots))=(F^{ad}(x_0),x_0,x_1,x_2,\ldots),
\]
and let $\pi:\inverselimit\to\Aff^{N,ad}_K\cup\{\infty\}$ be
\[
\pi((x_0,x_1,x_2,\ldots))=x_0.
\]
We can easily see that 
\begin{equation}\label{eq:fpipit}
F^{ad}\circ \pi=\pi\circ T.
\end{equation}
\begin{lemma}\label{Lem:Integralover}
Let $F^{*}(X_i)=F_i(X_1,\ldots,X_N)$. Then, the Tate algebra $K\la X_1,\ldots,X_N\ra$ is integral over $K\la F^{*}(X_1),\ldots, F^{*}(X_N)\ra$.
\end{lemma}
\begin{lemma}\label{Lem:rigidtorigid}
A point $x$ in $\Aff^{N,ad}_K$ is rigid if and only if $F^{ad}(x)$ is rigid.
\end{lemma}
\begin{proof}
We denote $M_x\coloneqq  \{f\mid |f(x)|=0\}\subseteq K\la X_1,\ldots, X_N\ra=R$. If $x$ is rigid, then $M_x$ is a maximal ideal. We observe that $F^*M_{F(x)}$ is $M_{x}\cap F^*R$, so $F^*M_{F(x)}$ is a maximal ideal of $F^*R$. Hence, the $M_{F(x)}$, the preimagine of $F^*M_{F(x)}$, is a maximal ideal of $R$, i.e. $F(X)$ is rigid. 

On the other hand, let $y=F(x)$, and assume $M_{y}$ is a maximal ideal of $R$, which implies $F^*M_{y}$ is a maximal ideal of $F^*R$, and, therefore, $M_x$ is an ideal with $M_x\cap F^*R=F^*M_{y}$. By Lemma~\ref{Lem:Integralover}, $M_x$ is a maximal ideal of $R$, which shows that $x$ is rigid.
\end{proof}

Because of this lemma, we will define the rigid point on the $\inverselimit$ as the following.
\begin{definition}
We say that a point in $\inverselimit$ is rigid if any coordinate, hence all coordinates, is rigid in $\Aff_K^{N,ad}$. 
\end{definition}
It would then be comfortable to say that the rigid points on the $\inverselimit$ projects to the rigid points on the adic space $\Aff^{N,ad}_K$, and the map is only surjective in general.

Now, we are going to construct a perfectoid space, and to do so, the first step is given a perfectoid $K$-algebraic $R^{perf}$ which is defined by the following way. Let $R^n=K\la X^{(n)}_1,\ldots,X^{(n)}_N\ra$ where $X^{(n)}_i$ are just some transcendental elements over $K$, and let $F_n:R^n\to R^{n+1}$ give by, for each $i$,
\[
F_n(X^{(n)}_i)=F_i(X_1^{(n+1)},\ldots,X_N^{(n+1)}).
\]
Then, we consider the following chain
\begin{center}
\begin{tikzcd}
R^0 \arrow[r,"F_0"] & R^1 \arrow [r, "F_1"] & R^2 \arrow[r, "F_2"] & \cdots \arrow[r,"F_{n-1}"]& R^{n}\arrow[r,"F_{n}"] & R^{n+1}\arrow[r,"F_{n+1}"] & \cdots
\end{tikzcd}
\end{center}
and let $R^{perf}=\widehat{\varinjlim_n R^n}$. Sometimes, it is convenient to write $R^{perf}$ as $K\la X_1^{(\infty)},\ldots,X_N^{(\infty)}\ra$.

Intuitively, the direct limit $\varinjlim_n R^n$ is an infinite algebraic extension of $R^0$ which is not necessary completed under the extended norm, a trivial example would be like $\bar{Q_p}$, so we complete the infinitely extended $K$-algebra to form a Banach $K$-algebra. Our next step is to show that this Banach $K$-algebra is the perfectoid $K$-algebra we need.
\begin{lemma}
The $R^{perf}$ is a perfectoid $K$-algebra with the set of bounded element
\[
R^{perf,\circ}=K^{\circ}\la X_1^{(\infty)},\ldots,X_N^{(\infty)}\ra
\]
\end{lemma}
\begin{proof}
We need first find out what the set of power bounded element is, and it is obvious that
\[
K^\circ\la X_1^{(\infty)},\ldots,X_N^{(\infty)}\ra\subseteq R^{perf,\circ}.
\]
Conversely, given a power series $f\in R^{perf,\circ}$, we can approach $f$ by a sequence of power series $\{f_n\}$ with $f_n\in R^n$. Then, since $\{f_n\}$ is Cauchy, there exists an integer $M$ such that $\|f_n-f_M\|<1$ for all $n\geq M$. It implies $f_n-f_M\in R^{(n),\circ}=K^\circ\la X^{(n)}_1,\ldots,X^{(n)}_N\ra$ for all $n\geq M$. Thus, $f-f_M\in\widehat{\varinjlim_n R^{n,\circ}}$. If $\|f_M\|\leq 1$, then $\|f\|\leq 1$. If $\|f_M\|>1$, $\|f_n-f_M\|\leq 1$ implies $\|f_n\|=\|f_M\|$ for all $n\geq M$. Thus, $\|f\|=\|f_M\|$. Similarly, since $f_n^m$ converges to $f^m$ for any integer $m$ and $\|f_n^m\|=\|f_n\|^m=\|f_M\|^m=\|f_M^m\|$ for all $n\geq M$, we have $\|f^m\|=\|f_M^m\|\to \infty$ as $m\to\infty$. Thus, $f$ is not bounded. Therefore, we can conclude that
\[
R^{perf,\circ}=K^{\circ}\la X_1^{(\infty)},\ldots, X_N^{(\infty)}\ra
\]
and also conclude that $R^{perf,\circ}$ is open and bounded.

Next, we should show that the Frobenius map  $\Phi:R^{perf,\circ}/\varpi\to R^{perf,\circ}/\varpi$ is surjective. It is because, for every $n\geq 0$, we have
\[
X_i^{(n)}\e (G_i(X_1^{(n+1)},\ldots,X_N^{(n+1)}))^p\mod \varpi.
\]
\end{proof}

Hence, it then is natural to define $\Aff^{N,perf}_K$ as the $Spa(R^{perf},R^{perf,\circ})$, and we finally can give a homeomorphism $\varphi:\Aff^{N,perf}_K\cup\{\infty\}\to \inverselimit$ to complete our claim.
We first note that, for any $f\in R^n$ and a valuation $\A\in\Aff^{N,perf}_K$, we can naturally take $|f(\A)|$ since $R^n$ has a natural embedding into $R^{perf}$ for all $n$. Let us denote this map as $\varphi^{(n)}: \Aff_{K}^{N,perf}\to \Aff_{K}^{N,ad}$. Beside, it is not hard to check $F^{ad}\circ\varphi^{(n+1)}=\varphi^{(n)}$ for any $n\geq 0$. Thus, we will define the "homeomorphism" $\varphi(\A)\coloneqq  (\varphi^{(0)}(\A),\varphi^{(1)}(\A),\varphi^{(2)}(\A),\ldots)\in\inverselimit$ for $\alpha\in\Aff^{N,perf}_K$ and simply define $\varphi(\infty)=(\infty,\infty,\ldots)$. The following lemma proves our claim.
\begin{lemma}\label{Lem:phihom}
The map $\varphi:\Aff^{N,perf}_K\cup\{\infty\}\to\inverselimit$ is homeomorphic.
\end{lemma}
\begin{proof}
We will construct an inverse map of $\varphi$ to show bijection. First, by Proposition 2.2 of \cite{Xie2016}(also see \cite{Huber1993}*{Proposition~3.9}), we have a natural morphism $\mu:\Aff^{N,perf}_K\to Spa(\varinjlim_n R^n,\varinjlim_n R^{n,\circ})$ which is a homeomorphism. Again, we will define $\mu^{-1}(\infty)=\{\infty\}$ to extend our map to the points in the infinity. To simplify the notations, we denote $\varinjlim_nR^n=B$ and $\varinjlim_nR^{n,\circ}=B^\circ$. Next, We let $\theta:\inverselimit\to Spa(B,B^\circ)\cup\{\infty\}$ by taking the "limit" of a point $(x_0,x_1,x_2,\ldots)\in \inverselimit$ and we denote $\theta((x_0,x_1,x_2,\ldots))=\lim_{n\to\infty}x_n$. We should show that this limit is a valuation of $B$ that has values less than one on $B^{\circ}$. More precisely, we take an $f\in B$, and we should note that the $f$ can be defined in $R^m$ for some $m$ by the definition of direct limit. Then, we define $f$ taking value at $\lim_{n\to\infty} x_n$ as
\[
|f(x_m)|
\]
This makes sense since $|f(x_m)|=|f(x_l)|$ for any $l\geq m$. By abusing the notation, we have
\[
|f(\lim_{n\to\infty} x_n)|=\lim_{n\to\infty} |f(x_n)|.
\]
It is easy to check that $\lim_{n\to\infty} x_n$ is a valuation. Moreover, as $f\in B^{\circ}$, we have
\[
|f(\lim_{n\to\infty} x_n)|=\lim_{n\to\infty}|f(x_n)|\leq 1.
\]
It is then easy to check that $\varphi\circ\mu^{-1}\circ\theta=id$ and $ \mu^{-1}\circ\theta\circ \varphi=id$. For continuity and open, since all the rational subset of $Spa(B,B^\circ)$ are defined over some $R^n$, it is easy to check that $\theta$ is continuous and open. Therefore, we can conclude that $\varphi$ is homeomorphic.
\end{proof}
We should give a map $F^{perf}:\Aff^{N,perf}_K\cup\{\infty\}\to\Aff^{N,perf}_K\cup\{\infty\}$ for the purpose that we want the following diagram commute.
\begin{center}
\begin{tikzcd}
\inverselimit&\arrow{l}{\varphi}\Aff_K^{N,perf}\cup\{\infty\}\\
\inverselimit\arrow{u}{T}&\arrow{l}{\varphi}\arrow{u}{F^{perf}}\Aff_k^{N,perf}\cup\{\infty\}
\end{tikzcd}
\end{center}
We define $F^{perf}$ as a map induced by $X_i^{(0)}\mapsto F_i(X_1^{(0)},\ldots, X_N^{(0)})$ and $X_i^{(n)}\mapsto X_i^{(n-1)}$ for all $n>0$ and for all $i$. We can easily check that this map satisfies our purpose and conclude these as the following lemma.
\begin{lemma}\label{FphiphiT}
Let $F^{perf}$, $T$ and $\varphi$ be as what we define above. Then, we have $F^{perf}\circ \varphi=\varphi\circ T$. Moreover, we have $\varphi(R(\Aff^{N,perf}_K\cup\{\infty\}))=R(\inverselimit)$.
\end{lemma}
\begin{proof}
The commutation is trivially true since we construct $F^{perf}$ for satisfying that purpose. Giving $x\in R(\Aff^{N,perf}_K)$, we have $M_x=\{f\in R^{perf}\mid |f(x)|=0\}$ maximal, and $M_{\varphi^{(0)}(x)}=M_x\cap R^0$ is trivially maximal in $R^0$, i.e. $\varphi(x)$ is a rigid point. Conversely, for any $y=(y_0,y_1,\ldots)\in R(\inverselimit)$, we note that $R^n/M_{y_n}=K$ for all $n$ since $K$ is algebraically closed. Moreover, the chain $R^0\to R^1\to R^2\to\cdots$ naturally induces a chain $R^0/M_{y_0}\to R^1/M_{y_1}\to R^2/M_{y_2}\to \cdots$. By taking the direct limit, we naturally have $\varinjlim_n R^n\to K$, which is a valuation on $\varinjlim_nR^n$. The $\varphi^{-1}(y)$ is the valuation on $\widehat{\varinjlim_n R^n}$ that uniquely extend from $\varinjlim_n R^n\to K$. Since the valuation group is a field, $\varphi^{-1}(y)$ is a rigid point, which means our proof is completed.
\end{proof}
\begin{remark}
Xie proves a much stronger result which does not assume $K$ is completed. However, the proof of Lemma~\ref{FphiphiT} illustrates the idea of Xie's proof.
\end{remark}
\subsubsection{A quick conclusion for the characteristic $0$}
Readers who do not want to go through all details can just read this section and only focus on the behaviors of the rigid points and the dynamical structure. If we only consider points on the valuation rings $K^\circ$, we can toss the abused infinity symbols away and conclude Lemma~\ref{tauFFtau}, Equation~\eqref{eq:fpipit} and Lemma~\ref{FphiphiT} as the following lemma.
\begin{lemma}\label{FinalLemmaforChar0}
We have the following diagrams commute.
\begin{center}
\begin{tikzcd}
\Aff_{K^{\circ}}^N \arrow{r}{\tau} & \Aff^{N, ad}_K & \arrow{l}{\pi} \varprojlim\limits_{F^{ad}}\Aff^{N,ad}_K &\arrow{l}{\varphi} \Aff^{N,perf}_K\\
\Aff_{K^{\circ}}^N \arrow{u}{F}\arrow{r}{\tau} & \Aff^{N, ad}_K \arrow{u}{F^{ad}}& \arrow{l}{\pi} \varprojlim\limits_{F^{ad}}\Aff^{N,ad}_K\arrow{u}{T} &\arrow{l}{\varphi} \Aff^{N,perf}_K\arrow{u}{F^{perf}}
\end{tikzcd}
\end{center}
Moreover, the closed points on $\Aff^N_{K^\circ}$ are one-to-one to the rigid points on $\Aff_K^{N,ad}$ which are projected from the rigid points on $\varprojlim_{F^{ad}}\Aff_K^{N,ad}$, and the rigid points on $\varprojlim_{F^{ad}}\Aff_K^{N,ad}$ are also one-to-one to the rigid points on $\Aff_K^{N,perf}$.
\end{lemma}
\subsection{The characteristic $p$}
We use P. Scholze's perfectoid space theory to construct the tilt of the perfectoid space $\Aff^{N,perf}_K$ in this section. We consider the following chain
\begin{center}
\begin{tikzcd}
R^{0,\flat}\arrow{r}{F^{0,\flat}}& R^{1,\flat}\arrow{r}{F^{1,\flat}}& R^{2,\flat}\arrow{r}{F^{2,\flat}} &\cdots\arrow{r}{F^{n-1,\flat}}& R^{n,\flat}\arrow{r}{F^{n,\flat}} &R^{n+1,\flat}\arrow{r}{F^{n+1,\flat}}&\cdots
\end{tikzcd}
\end{center}
where $R^{n,\flat}=K^\flat\la X_1^{(n)},\ldots,X_N^{(n)}\ra$ and $F^{n,\flat}(X_i^{(n)})=\sum_{j=1}^{m_i} (X_i^{(n)})^{q_j}$ for each $i$. Then, we denote $R^{perf,\flat}$ as the $\widehat{\varinjlim_n R^{n,\flat}}$. It then follows that the tilt of the perfectoid space $\Aff^{N,perf}_K$ is $Spa(R^{perf,\flat},R^{perf,\flat,\circ})$, denoted by $\Aff^{N,perf}_{K^\flat}$. The tilt of the function $F^{perf}$, denoted by $F^{perf,\flat}$, then is induced by the maps $X_i^{(0)}\mapsto \sum_{j=1}^{m_i} (X_{i}^{(0)})^{q_j}$ and $X_i^{(n)}\mapsto X_{i}^{(n-1)}$ for all $n>0$ and for all $i$.

The theory of perfectoid theory says we have a homeomorphism $\rho:\Aff^{N,perf}_K\to \Aff^{N,perf}_{K^\flat}$ and the dynamical systems $(\Aff_K^{N,perf},F^{perf})$ and $(\Aff_{K^\flat}^{N,perf},F^{perf,\flat})$ is isomorphic by $\rho$, i.e.
\begin{equation}\label{Eq:FrhorhoF}
F^{perf,\flat}\circ\rho=\rho\circ F^{perf}.
\end{equation}

The rests are pretty straight forward. Let $F^{\flat}:\Aff^N_{K^\flat}\to F\Aff^N_{K\flat}$ be given by $F^{\flat}=(F_1^{\flat},\ldots,F_N^\flat)$ where
\[
F_i^{\flat}(X_1,\ldots,X_N)=\sum_{j=1}^{m_i}X_i^{q_j},
\]
By following the same moral, the $F^{\flat,ad}:\Aff^{N,ad}_{K^\flat}\cup\{\infty\}\to\Aff^{N,ad}_{K^\flat}\cup\{\infty\}$  is a morphism induced by $X_i\mapsto \sum_{j=1}^{m_i}X_i^{q_j}$ for each $i$. We also have the natural embedding $\tau^{b}:\Aff_{K^\flat}^N\to \Aff_{K^\flat}^{N,ad}\cup\{\infty\}$.  Then, recall that the proof of Lemma~\ref{tauFFtau} Lemma~\ref{Lem:rigidtorigid}, Lemma~\ref{Lem:phihom} and Lemma~\ref{FphiphiT} does not use any fact about the characteristic of the field $K$ , so those lemmas are also true over $K^\flat$. Let us define $\pi^b:\inverselimitb\to\Aff^{N,ad}_{K^\flat}\cup\{\infty\}$ as $\pi^{\flat}((x_0,x_1,x_2,\ldots))=x_0$, $T^{\flat}:\inverselimitb\to\inverselimitb$ as $T((x_0,x_1,x_2,\ldots))=(F^{\flat,ad}(x_0),x_0,x_1,x_2,\ldots)$ and finally the map $\varphi^\flat:\Aff^{N,perf}_{K^\flat}\cup\{\infty\}\to \inverselimitb$ as the map inducing by the natural embeddings $K^\flat\la X_1^{(n)},\ldots X_N^{(n)}\ra\hookrightarrow K^\flat\la X_1^{(\infty)},\ldots, X_N^{(\infty)}\ra$ for all $n$. We give a lemma that is analogue to Lemma~\ref{FinalLemmaforChar0} to conclude this section.   
\begin{lemma}\label{FinalLemmaforCharp}
We have the following diagrams commute.
\begin{center}
\begin{tikzcd}
\Aff_{K^{\flat,\circ}}^N \arrow{r}{\tau^{\flat}} & \Aff^{N, ad}_{K^{\flat}} & \arrow{l}{\pi^{\flat}} \varprojlim\limits_{F^{\flat,ad}}\Aff^{N,ad}_{K^{\flat}}&\arrow{l}{\varphi^{\flat}} \Aff^{N,perf}_{K^{\flat}}\\
\Aff_{K^{\flat,\circ}}^N \arrow{u}{F^{\flat}}\arrow{r}{\tau} & \Aff^{N, ad}_{K^{\flat}} \arrow{u}{F^{\flat,ad}}& \arrow{l}{\pi^{\flat}} \varprojlim\limits_{F^{\flat,ad}}\Aff^{N,ad}_{K^{\flat}}\arrow{u}{T^{\flat}} &\arrow{l}{\varphi^{\flat}} \Aff^{N,perf}_{K^{\flat}}\arrow{u}{F^{perf,\flat}}
\end{tikzcd}
\end{center}
Moreover, the closed points on $\Aff^N_{K^{\flat,\circ}}$ are one-to-one to the rigid points on $\Aff_{K^{\flat}}^{N,ad}$ which are projected from the rigid points on $\varprojlim_{F^{\flat,ad}}\Aff_{K^{\flat}}^{N,ad}$, and the rigid points on $\varprojlim_{F^{\flat,ad}}\Aff_{K^{\flat}}^{N,ad}$ are also one-to-one to the rigid points on $\Aff_{K^{\flat}}^{N,perf}$.
\end{lemma}
\section{Restricting on Periodic Points}

\subsection{Passing by Reduction}
The reduction map is naturally defined on a projective space $\Proj^N_K$ not on the the affine space, but we can still construct a reduction map by identifying the affine space to a particular affine peace of the projective space and use the reduction map on the projective space to give a reduction map on the affine space.

In the followings, we define $\bar{\rho}:\Proj_k^N\to\Aff_k^N\cup\{\infty\}$ by
\begin{align*}
\bar{\rho}:[\bar{x_0}:\bar{x_1}:\cdots:\bar{x_N}]&\mapsto(\bar{x_1}\bar{x_0}^{-1},\ldots,\bar{x_N}\bar{x_0}^{-1})\quad\mbox{if }\bar{x_0}\neq \bar{0}\mbox{, and}\\
[\bar{0}:\bar{x_1}:\cdots:\bar{x_N}]&\mapsto \infty
\end{align*}
and define $\rho^{-1}:\Aff_K^N\to\Proj_K^N$ by
\[
\rho^{-1}:(x_1,\ldots,x_N)\to[1:x_1:\cdots:x_N].
\]
The reduction map $red:\Proj_K^N\to\Proj_k^N$ is defined by
\[
red_{\Proj}:[x_0:x_1:\cdots:x_N]\to[\bar{x_0}:\bar{x_1}:\cdots:\bar{x_N}]
\]
where we need $\|[x_0:x_1:\cdots:x_N]\|\coloneqq  \max\{|x_i|\}=1$ and $\bar{x_i}\coloneqq  x_i\mod K^{\circ\circ}$. Then the reduction map on an affine space is naively defined as $red_{\Aff}\coloneqq  \bar{\rho}\circ red\circ\rho^{-1}:\Aff_K^N\to\Aff_k^N\cup\{\infty\}$. We should omit the index from the $red_{\Aff}$ and just denote by $red$ if the domain is clear in our discussion. The following lemma shows that this definition makes sense of our purpose.
\begin{lemma}\label{Lem:redffred}
Let $\bar{F}=F\mod K^{\circ\circ}$ where $F$ is a restricted lift of $p$th power . Then, we have the following diagram commutes,

\begin{center}
\begin{tikzcd}
\Aff_K^N\arrow{d}{red}\arrow{r}{F}&\Aff_K^N\arrow{d}{red}\\
\Aff_k^N\cup\{\infty\}\arrow{r}{\bar{F}}&\Aff_k^N\cup\{\infty\}
\end{tikzcd}
\end{center}
i.e. $F\circ red=red\circ\bar{F}$.
\end{lemma}
\begin{proof}
It directly follows by Lemma~\ref{keylemma}
\end{proof}
The main purpose of this section is to show that the reduction map $red$ gives an injective map from the periodic points of $F$ over K to the periodic points of $\bar{F}$ over $k$. More precisely, let us give the following definition.
\begin{definition}
Let $f:X\to X$ be an automorphism, and depending on the context, we denote $R(X)$ to be the closed or rigid points of $X$. Then, we define the $\Per(f)=\{\A \in R(X)\mid f^n(\A)=\A\text{ for some }n\}$.
\end{definition}
This is not a standard notation. However, we only care about rigid points and closed points in this paper. Moreover, the fields $K$ and $K^\flat$ in our context are algebraically completed which makes our notations less ambiguous. The conclusion of this section is the following lemma.
\begin{lemma}
The reduction map $red:\Aff^N_K\to\Aff^N_k$ restricts on the $\Per(F)$ is bijective to $\Per(\bar{F})$ where $F$ is a restricted lift of $p$th power and $\bar{F}=(\bar{F_1},\ldots,\bar{F_2})$ where each $\bar{F_i}$ is the $F_i\mod \varpi$.
\end{lemma}
\begin{proof}
First, along the argument in Lemma~\ref{Lem:redffred}, it is easy to prove $\Per(F)\cap\Aff^N_K\setminus\Aff^N_{K^\circ}=\emptyset$. Moreover, by the same lemma, it is clear that $red(\Per(F))\subseteq \Per(\bar{F})$ since $red\circ F=\bar{F}\circ red$. For the surjectivity, for every periodic point $\bar{\A}\in\Aff_k^N$, $\bar{\A}$ is a solution of the equation $\bar{F}^n=(X_1,\ldots,X_N)$, and it naturally implies that there is a solution $\A\in\Aff_K^N$ of the equation $F^n=(X_1,\ldots,X_N)$ for which $red(\A)=\bar{\A}$. Finally, we are going to show bijection by proving that $F$ is attracting on $red^{-1}(y)$ for every $y\in\Per(\bar{F})$ 

Let $\textrm{x}_0\coloneqq  (\bar{x_{0,1}},\ldots,\bar{x_{0,N}})\in red^{-1}(y)\cap\Per(F)$. We want to show that every other point $\textrm{x}\in red^{-1}(y)\setminus\{\textrm{x}_0\}$ are attracting to $\textrm{x}$. By induction, it is not hard to show that, for all $i$, the $i$th coordinate of $F^m$ for all $m$ is still of the form
\[
\sum_{j=1}^{m_i} X_i^{q_j}+\varpi f_i(X_1,\ldots,X_N)
\]
where, of course, $q_j$s and $f_i$ are different from what we give in the beginning. Since $\textrm{x}_0$ is a fixed point of $F^m$ for some $m$, the induction step allows us to assume $m=1$ without loss of generality. Rewrite $\mathrm{x}$ by $(x_{0,1}+r_1,\ldots,x_{0,N}+r_N)$ for some $r_i\ K^{\circ\circ}$, and then we have

\begin{align*}
|F_i(\mathrm{x}_0+r)-F_i(\mathrm{x}_0)|&=|G_i(\mathrm{x}_0+r)^p+\varpi f_i(\textrm{x}_0+r)-G_i(\mathrm{x}_0)^p-\varpi f_i(\mathrm{x_0})|\\
&\leq\max\{|G_i(\mathrm{x}_0+r)-G_i(\mathrm{x}_0)|^p,\varpi|f_i(\mathrm{x}_0+r)-f_i(\mathrm{x}_0)|\}\\
&\leq\max\{\max_{i}|r_i|^p,\varpi\max_i|r_i|\} <\max_i\{r_i\}\\
&=\|\textrm{x}_0-\textrm{x}\|.
\end{align*}

Hence, we conclude that $F$ is attracting on $red^{-1}(y)$, and the injection follows.
\end{proof}
\subsection{Passing by tilt}
First, we note that $\tau$ (resp. $\tau^\flat$) and $\varphi$ (resp. $\varphi^\flat$) induces a bijection on the $\Per(F)\to \Per(F^{ad})$ (resp. $\Per(F^{\flat})\to\Per(F^{\flat,ad})$) and $\Per(F^{perf})\to\Per(T)$ (resp. $\Per(F^{perf,\flat})\to \Per(T^\flat)$) respectively, and the connection between $\Per(F^{ad})$ (resp. $\Per(F^{\flat,ad})$) and $\Per(T)$ (resp. $\Per(T^\flat)$) can be built naturally by the following way. We define $\chi:\Per(F^{ad})\to \Per(F^{ad})$ (similarly, $\chi^\flat$) as the following, for any $\A\in\Per(F^{ad})$,
\[
\chi:\A\mapsto (\A, F^{n-1}(\A), F^{n-2}(A),\ldots, F^2(\A), F(\A), \A, F^{n-1}(\A),\ldots) 
\]
where the integer $n$ is the period of $\A$, i.e. $(F^{ad})^m(\A)\neq \A$ for all integer $m<n$.

This definition naturally gives a bijection from $\Per(F^{ad})$(resp. $\Per(F^{\flat,ad})$) to $\Per(T)(resp. \Per(T^\flat))$, and, therefore, by Lemma~\ref{FinalLemmaforChar0}, Lemma~\ref{FinalLemmaforCharp} and Equation~\eqref{Eq:FrhorhoF}, we have a bijection
\[
\iota\coloneqq   \pi^{\flat}\circ\varphi^{\flat}\circ\rho\circ\varphi^{-1}\circ\chi:\Per(F^{ad})\to \Per(F^{\flat,ad}).
\]
Observe that for every point $\A\in \Per(F)$, we have $\red(\A)= \red^{\flat}\circ(\tau^{\flat})^{-1}\circ\iota\circ\tau(\A)$ where $\red^{\flat}:\Aff^N_{K^\flat}\to \Aff_k^N$ is the natural reduction, and so we have
\[
\phi\coloneqq  \eta\circ\red=\eta\circ\red^{\flat}\circ(\tau^{\flat})^{-1}\circ\iota\circ\tau=(\tau^{\flat})^{-1}\circ\iota\circ\tau
\]
on $\Per(F)$ where $\eta:\Aff_k^N\to\Aff_{K^\flat}^N$ is the natural embedding with $\eta\circ\red^\flat= id$ on the $\Aff^{N}_{k}\subset\Aff^{N}_{K^\flat}$. In words, the periodic points of $F$ passing from the left is equal to passing from the right.
\section{The main theorem}
\subsection{proof of the dynamical Manin-Monford conjecture}
\begin{reptheorem}{the dynamical Manin-Mumford conjecture}
Let $V$ be a variety of $\Aff_K^N$ and $F$ is a restricted lift of $p$-th power, and assume that $\bar{F}^l$ is a finite surjective morphism of $red(V)\coloneqq \bar{V}$ for some integer $l\geq 1$, i.e. $\bar{F}^l:\bar{V} \to \bar{V}$. Then, if $V\cap \Per(F)$ is dense in $V$, then $V\cap\Per(F^l)=V\cap\Per(F)$ is periodic; hence, $V$ is periodic.
\end{reptheorem}
Before proving Theorem~\ref{Inverse Mordelll-Lang conjecture}, there is an interested conclusion following from the theorem.
\begin{corollary}\label{cor:1}
If the $\# \Per(f)\cap\Per(g)$ is infinite where $f, g:\Aff^N_K\to\Aff^N_K$ are restricted lift of $p$th power with $\bar{f}=\bar{g}$, then $f=g$.
\end{corollary}
\begin{proof}[Proof of Corollary~\ref{cor:1}]
Because of the assumption, we have $\Per((f,g))\cap \Delta$ where $\Delta:=\{(x,x)\in\Aff_K^{2N}\mid x\in \Aff_K^N\}$ is dense in $\Delta$. Since $K$ is algebraically closed, its residue field $k$ is also algebraically closed. Therefore, we have $f,g:\Aff_k^N\to \Aff_k^N$ are surjective maps which implies, $(\bar{f},\bar{g}):\bar{\Delta}\to\bar{\Delta}$ is a surjective morphism of the diagonal. Hence, by Theorem~\ref{Inverse Mordelll-Lang conjecture}, $\Delta$ is periodic over $(f,g)$ which implies $f(x)=g(x)$ for all $x\in\Aff_K^N$ which is our conclusion.
\end{proof}

In the rest of this section, we need $|t|=|\varpi|$ and suppose the defining equations of $V$ are $H_j(X_1,\ldots,X_N)=0$, $j=1,\ldots,m$ where $H_j$s are polynomials. Since the set of zeros will not change by multiplying the polynomials with some nonzero constants, we can assume $\|H_j\|=1$ for all $j$. Set
\[
V^{ad}\coloneqq  \{x\in \Aff^{N,ad}_K\mid |H_j(x)|=0, j=1,\ldots, m\},
\]
and we have $R(V^{ad})\coloneqq  V^{ad}\cap R(\Aff^{N,ad}_K)= \tau(V)$.

We need one of Xie's lemmas to prove the main theorem.
\begin{lemma}[Xie]\label{Xie's lemma}
Let $h\in K^\flat[X_1,\ldots,X_N]$ be a polynomial with Gauss norm $1$. Suppose that for every point $x\in \iota(\Per(F^l)\cap V^{ad})$, we have $|h(x)|\leq |t|^s$, where $s\in \Z^+$. Then for every point $y\in \tau^\flat(\eta(\bar{V}))$, we have $|h(y)|\leq |t|^s$. 
\end{lemma}

The following theorem in \cite{Amerik2011} imply Xie's Lemma.
\begin{theorem}\label{thm:Amerik}
Let $X$ be an algebraic variety over $k$ and $f:X\to X$ be an finite surjective morphism. Then the subset of $X$ consisting of periodic points of $f$ is Zariski dense in $X$.
\end{theorem}
\begin{proof}[Proof of Xie's Lemma]
We have a natural map
\[
(K^{\flat,\circ})^N\to (K^{\flat,\circ}/(t^s))^N=\Aff^{N}_{K^\flat/(t^s)}(K^\flat/(t^s))
\]
defined by $(x_1,\ldots,x_N)\mapsto (\bar{x_1},\ldots,\bar{x_N})$ where we denote $\bar{x_i}$ as the $x_i\mod t^s$. We denote $H\mod t^s$ by $\bar{H}$. Then, since $\bar{F}^l:\bar{V}\to\bar{V}$ is a surjective morphism, $red(\Per\cap V)$ dense in $\bar{V}$ follows from Theorem~\ref{thm:Amerik}. Moreover, we have
\[
\overline{\overline{\eta(\bar{V})}}=\overline{\overline{\phi(\Per\cap V)}}=\overline{\overline{\iota(\Per\cap V)}}.
\]
We use the double overline to denote that we take Zarisky closure to the set in its belonging space.

By assumption, we have $|\bar{H}(\bar{x})|=0$ for every point $x\in \iota(\Per\cap V)$. Observe that
\[
\overline{\iota(\Per\cap V)}=\phi(\Per\cap V)\times_{\Spec k}\Spec(K^{\flat,\circ}/(t^s))
\]
is Zarisky dense in $\overline{\overline{\eta(\bar{V})}}\times_{\Spec k}\Spec(K^{\flat,\circ}/(t^s))$. It follows that
\[
\overline{\overline{\eta(\bar{V})}}\times_{\Spec k}\Spec(K^{\flat,\circ}/(t^s))\subseteq\{\bar{H}=0\}
\]
Therefore, we have 
\[
\eta(\bar{V})\times_{\Spec k}\Spec(K^{\flat,\circ}/(t^s))\subseteq\{\bar{H}=0\},
\]
i.e. $H(y)\e 0\mod t^s$ for all $y\in \eta(\bar{V})$, and immediately it implies $|H(\tau^{\flat}(y))|\leq |t|^s$ by the definition of $\tau^{\flat}$.
\end{proof}

Now, we are well prepared for showing our first main theorem.

\begin{proof}[Proof of Theorem]
Let us denote
\[
\mathcal{O}\coloneqq \bigcup_{x\in\Per(F)\cap V}\orbit{F^l}{x}.
\]
Since $F^l$ is bijective on the $\orbit{F^l}{x}$ for a periodic point $x$, $F^l(\mathcal{O})=\mathcal{O}$. We claim $\mathcal{O}\subseteq V$, where $\mathcal{O}$ is dense in. To prove the claim, we need
\[
|H_j(s)|=0\text{ for all }j\text{ and all }s\in\mathcal{O},
\]
or equivalently $|H_j(\varphi^{-1}\circ\chi(s))|=0$.

For any $x\in\Per(F)\cap V$, we define $z\coloneqq  \varphi^{-1}\circ\chi(x)$ and $z^\flat\coloneqq \rho(z)$ . Fixing a $j$, we apply the approximation lemma to $H_j$. For every $c\geq 0$ and $\varepsilon=1/2>0$, there exists a $g_{c}\in R^{\flat,perf}$ such that we have
\begin{equation}\label{inequality1}
|H_j-g^{\sharp}_c(v)|\leq |\varpi|^{1/2}\max\{|H(v)|,|\varpi|^c\}<1\qquad\text{for all }v\in\Aff^{N,perf}_K.
\end{equation}
Then, $\|g_c^\sharp\|\leq 1$ follows $\|H_j\|\leq 1$. The must importantly,  Inequality~\ref{inequality1} becomes
\[
|g_c^\sharp(z)|=|g_c(z^\flat)|\leq |t|^{c+1/2}
\]
because of $|H(z)|=0$.

We can approach $g_c$ by a polynomial $h_c$ in $R^{\flat,n}$ for some $n=lk$ where $k$ is a large enough integer, i.e. $h_c$ satisfies $\|g_c-h_c\|\leq |t|^c$. We can express $h_c$ as a polynomial in terms of $X_1^{(n)},\ldots,X_N^{(n)}$ by replacing $X_i^{(k)}$ in terms of $X_i^{(n)}$ for all $k<n$. Note that the replacement does not change anything, so we still have $|h_c(z^\flat)|\leq |t|^c$. Therefore, we conclude that 
\[
|h_c(s^\flat)|\leq |t|^c\quad\text{for all }s^\flat\in\mathcal{O}^{ad}\coloneqq \bigcup_{ y\in\Per((F^{ad})^l)\cap \tau^{\flat}(\bar{V})}\orbit{F^{ad}}{y}\subseteq \tau^{\flat}(\bar{V})
\]
by Xie's lemma. The construction of the perfectoid spaces naturally implies  $|h_c(s')|=|h_c(s)|\leq |t|^c$
\[
\forall s'\in \varphi^{-1}\{((F^{ad})^{n}(s),(F^{ad})^{n-1}(s),(F^{ad})^{n-2}(s),\ldots, s,s_1,s_2,\ldots)\mid F^{ad}(s_1)=s\text{ and }F^{ad}(s_{n+1})=s_n\quad\forall n>0\}.
\]
In particular, $|h_c(s')|\leq |t|^c$ for all $s'\in(\varphi^{\flat})^{-1}\circ\chi(\mathcal{O}^{ad})$. Thus, we conclude that
\[
|g_c(s')|=|h_c(s')+(g_c(s')-h_c(s'))|\leq |t|^c.
\]

Finally, we have $|g_c^\sharp(\rho^{-1}(s'))|=|g_c(s')|\leq |t|^c$. Let $s''=\rho^{-1}(s')$. Following by Inequality~\ref{inequality1} and the remark after the approximation lemma, we conclude that the inequality
\[
|H_j-g_c^\sharp(s'')|\leq |t|^{1/2}\max\{|g_c^\sharp(s'')|,|t|^c\}=|t|^{c+1/2}.
\]
implies 
\begin{equation}\label{inequality5}
|H_j(s'')|=|g_c^\#(s'')-(H_j-g_c^{\#}(s''))|\leq\max\{|H_j(s'')|,|H_j-g_c^{\#}(s'')|\}\leq\max\{|t|^c,|t|^{c+1/2}\}=|t|^c= |\varpi|^c.
\end{equation}
Note that $\rho^{-1}\circ(\varphi^{\flat})^{-1}\circ\chi^{\flat}:\mathcal{O}^{ad}\to\varphi^{-1}\circ\chi(\mathcal{O})$ is a bijection, so Inequality~\ref{inequality5} is  true for all $s\in\mathcal{O}$. Since $c$ is arbitrary, and also the statement is true for all $j$, we conclude that $\mathcal{O}\subset V$, which is our desire result.
\end{proof}
\subsection{Proof of the Dynamical Tate-Voloch Conjecture}
\begin{reptheorem}{dynamical tate-voluch conjecture}
Let $F$ be a restricted lift of $p$th power. Then, for any variety $V$, there exists an $\varepsilon>0$ such that, for all $x\in \Per(F)$, either $x\in V$ or $d(x,V)>\varepsilon$.
\end{reptheorem}
The following proposition gives an idea about how to do this question.
\begin{proposition}
Given a power series $g\in K^{\flat,\circ}\la X_1,\ldots,X_N\ra$ of the form
\[
g=\sum_{i=0}^\infty g_{i}u^i
\]
where $u\in K^{\flat,\circ}$ has valuation $|u|<1$ and where the $g_i$ are taking from $k\la X_1,\ldots, X_N\ra$. Then, there exists an $\varepsilon$ depended on $g$ such that $g(x)=0$ or $|g(x)|> \varepsilon$ for all $x\in k^N$.
\end{proposition}
\begin{proof}
We note two things. First, we have either $g_i(x)=0$ or $|g_i(x)|=1$ for any $x\in k^N$. Secondly, the polynomial ring $k[X_1,\ldots,X_N]$ is Noetherian, so there exists an integer $M$ for which $I=\la g_i\ra_{i=0,1,2,\ldots}$ is generated by $g_0,g_1,\ldots, g_M$. Hence, if we have $|g(x)|<|u|^{M}$, then we have $g_i(x)=0$ for $i=0,1,2\ldots, M$; hence, $g_i(x)=0$ for all $i$. The $|u|^M$ is our desired $\varepsilon$.
\end{proof}
This proposition indicates that we should approach the power series $g_c$ in the approximation lemma not only by a polynomial but by a polynomial in the form given in the proposition.
\begin{proof}[Proof of Theorem~\ref{dynamical tate-voluch conjecture}]
Given an polynomial $H\in R^0\subset R^{perf}$ with $\|H\|=1$, for $\varepsilon=1/2$ and any $c>0$, in particular we want $c$ to be an integer, there exists an element $g_c\in R^{\flat,perf}$ such that
\[
|H-g_c^\sharp(x)|\leq |\varpi|^{1/2}\max\{|H(x)|,|\varpi|^c\}<1
\]
for all $x\in\Aff^{perf}_{K}$. It implies $\|g_c\|= 1$, so we can approach $g_c$ by elements in $R^{\flat,n,\circ}$. More precisely, for a large enough $n$, there exists a polynomial $h_c\in R^{\flat,n,\circ}$ such that $\|g_c-h_c\|<|t|^{1/2+c}$. We write $h_c$ as
\[
h_c(X^{(n)}_1,\ldots,X^{(n)}_N)=\sum_{i\coloneqq (i_1,\ldots,i_N)\in I} c_i{X^{(n)}_1}^{i_1}\cdots {X_N^{(n)}}^{i_N}
\]
where for all but finitely many $c_i$ are zero. Since we only have finitely many $c_i\in K^{\flat,\circ}$, we can find a field extension $E=k((u))$ of $k((t))$ for some $|u|=|t|^{1/\deg(u)}$ such that there exist $c'_i\in E^{\circ}$ with $|c_i-c'_i|<|t|^{1/2+c}$. Let $h'_c=\sum_{i\in I}c_i'{X^{(n)}_1}^{i_1}\cdots {X_N^{(n)}}^{i_N}$, and we can rewrite $h'_c$ as the form
\[
h'_c(X^{(n)}_1,\ldots,X^{(n)}_N)=\sum_{i=0}^\infty f_{c,i}(X^{(n)}_1,\ldots,X^{(n)}_N) u^i\qquad \text{for some }f_{c,i}\in k[X^{(n)}_1,\ldots,X^{(n)}_N].
\]
We then choose $h^0_c=\sum_{i=0}^{M_c}f_{c,i} u^i$ where $M_c$ is the largest integer having $|u^i|>|t|^c$. This means $\|h_c-h^0_c\|\leq |t|^c$.

We let $I=\la f_{c,i}\mid \forall c,i\ra$ which is an ideal of the Noetherian ring $k[X_1^{(n)},\ldots,X_N^{(n)}]$, so there exists an integer $M$ for $I=\la f_{c,i}\mid c\leq M,\ \forall i\ra$. Now, we need a lemma.
\begin{lemma}\label{thelemma}
We fix a valuation $x\in \Per(F^{perf})$. Then, $|H(x)|=0$ if and only if $f_{c,i}(x^\flat)=0$ for all $c\leq M$ and for all $i$.
\end{lemma}
\begin{proof}[Proof of Lemma~\ref{thelemma}]
We first prove from the left to the right. Because of $|H(x)|=0$, for any $c$,
\[
|g_c(x^\flat)|=|g_c^\sharp(x)|\leq |t|^{1/2+c}.
\]
Then, we have
\[
|h_c(x^\flat)|=|g_c(x^\flat)-(g_c(x^\flat)-h_c(x^\flat))|\leq |t|^{1/2+c},	
\]
and use the same argument again to give
\begin{equation}\label{inequality6}
|h_c^0(x^\flat)|=|h_c(x^\flat)-(h_c(x^\flat)-h_c^0(x^\flat))|\leq\max\{|h_c(x^\flat)|,\|h_c-h_c^0\|\}\leq |t|^c.
\end{equation}
Since $x^{\flat}\in \Per(F^{\flat,perf})$ is correspondent to the point $(x_0^\flat,x_1^\flat,x_2^\flat,\ldots)\in\varprojlim_{F^{\flat,ad}}\Aff^{N,ad}_{K^\flat}$ where each $x_i^\flat$ comes from the natural embedding $k^N\xhookrightarrow{}\Aff^N_{K^\flat}\xhookrightarrow{\tau^\flat}\Aff^{N,ad}_{K^\flat}$, we have $|f_{c,i}(x^\flat)|=|f_{c,i}(x_n^\flat)|=1$ or $0$. Inequality~\ref{inequality6} forces $|f_{c,i}(x_n^\flat)|=0$ for all $i=0,1,\ldots, M_c$. Since $c$ is arbitrary, we prove one direction.

Conversely, if $f_{c,i}(x^\flat)=0$ for all $c\leq M$ and $i$, then $f_{c,i}(x^\flat)=0$ for all $c$ and $i$. It then implies, for arbitrary $c$,
\[
|h_c(x^\flat)|=|h_c(x^\flat)-h_c^0(x^\flat)|\leq |t|^c;
\]
therefore, $|g_c^\sharp(x)|=|g_c(x^\flat)|\leq |t|^c$. So, we have
\begin{align*}
|H(x)|=&|g_c^\sharp(x)-(g_c^\sharp(x)-H(x))|\\
&\leq\max\{|g_c^\sharp(x)|,|g_c^\sharp(x)-H(x)|\}\\
&\leq\max\{|\varpi|^c,|\varpi|^{1/2}\max\{|g_c^\sharp(x)|,|\varpi|^c\}\}\\
&\leq|\varpi|^c.
\end{align*}
Since $c$ is arbitrary, $|H(x)|=0$.
\end{proof}

The rest are pretty straightforward. If we have $x\in \Per(F)$ and $H(x)=0$, then we are done. Otherwise, redefine $x$ as $\varphi^{-1}\circ\chi(x)$, and we have $|H(x)|=|H(\pi\circ\varphi(x))|\neq 0$. By Lemma~\ref{thelemma}, $|f_{c_0,i}(x^\flat)|=1$ for some $c_0<M$ and for some $i$, i.e. $|h_c^0(x^\flat)|>|t|^c$ which implies $|h_c(x^\flat)|>|t|^c$, and also $|g_c^\sharp(x)|=|g_c(x^\flat)|>|t|^c$. Since $\max\{|H(x)|,|\varpi|^c\}=\max\{|g_c^\sharp(x)|,|\varpi|^c\}$ and $|g_c^\sharp(x)|>|\varpi|^c$, we conclude that $|H(\pi\circ\varphi(x))|=|H(x)|=|g_c^\sharp(x)|> |\varpi|^c\geq |\varpi|^M=:\varepsilon$.
\end{proof}
\subsection{Proof of inverse dynamical Manin-Munford Conjecture}
For the inverse dynamical Manin-Munford conjecture, I need to assume a strong condition which I will explain later. 
\begin{definition}
For any given coherence backward orbit $\{a_i\}\subset \Aff_K^{N}\simeq \Aff_K^{N,ad}(K)$, we then can associate $\{a_i\}$ with a map $\chi_{\{a_i\}}:\Aff^{N,ad}_K\cup \{\infty\}\to\inverselimit$ by
\[
\chi_{\{a_i\}}:a_n\mapsto (a_n,a_{n+1},a_{n+2},\ldots)
\]
and define $b_i\coloneqq \pi^{\flat}\circ\varphi^{\flat}\circ\rho\circ\varphi^{-1}\circ\chi_{\{a_i\}}(a_i)\in \Aff_{K^\flat}^{N,ad}(K^\flat)\simeq \Aff_{K^\flat}^N$. 
\end{definition}
Note that if we further assume that $b_0$ is algebraic, i.e. $b_0$ belongs to $\Aff_{K^\flat}^{N,ad}(k)$, then $b_i$ is algebraic for all $i$. This observation motivates a nontrivial assumption given in Theorem~\ref{Inverse Mordelll-Lang conjecture}.
\begin{reptheorem}{Inverse Mordelll-Lang conjecture}
Let $\{a_n\}$ be a coherent backward orbit of $a_0$ by a restricted lift of $p$th power $F$. Assume that there exists a subsequence $\{a_{n_i}\}$ contained in an irreducible subveriety $V$ with $\bar{F}^{l}(\bar{V})=\bar{V}$ for some integer $l$. We further assume that $b_0$ is algebraic. Then, we have $\{n_i\}$ is arithmetic progressive.
\end{reptheorem}
\begin{proof}
If $\{b_{n_i}\}$ is a finite set, then it implies the backward orbit is periodic. Therefore, without lose of generality, we can assume $\{b_{n_i}\}$ is an infinite set. Let $Z$ be the Zarisky closure of $\{b_{n_i}\}\subset \Aff_{K^{\flat}}^{N,ad}(K^\flat)$, and we should note that $Z\subseteq \tau^{\flat}\circ\eta(\bar{V})$. $Z=\tau^{\flat}\circ\eta(\bar{V})$ follows from the assumption that $V$ is irreducible. For convenience, we denote $(F^{\flat,ad})^i(b_{0})$ by $b_{-i}$. 

We define $L_t\coloneqq \{n_i\mid n_i\e t\mod l\}$ which partitions the index set $\{n_i\}$ into disjoint subsets. Since $\{n_i\}$ is infinite, one of the $L_t$ is infinite. Fix the $t$, and we want to show that $a_{i}\in V$ for all $i$ satisfying $i\e t\mod l$. At first, because of $(F^{\flat,ad})^l(Z)=Z$, we immediately have $b_{t+jl}\in Z$ for all $j$. 

The variety $V$ is defined by polynomials $H_j(X_1,\ldots,X_N)\in K[X_1,\ldots,X_N]$ for $j=1,\ldots, m$ with $\|H_j\|=1$ for all $j$. Fixing a $j$ and for any $c>0$, there exists $g_c\in R^{\flat,perf}$ such that we have
\[
|H_j-g_c^\sharp(v)|\leq|\varpi|^{1/2}\max\{|H(v)|,|\varpi|^c\}<1\quad\mbox{for all }v\in \Aff_K^{N,perf}.
\]
Following the proof of Theorem~\ref{Inverse Mordelll-Lang conjecture} exactly, we can approach $g_c$ by a polynomial $h_c$ in $R^{\flat, n}$ for some large enough $n=lk$ and have $|h_c(b_{n_i})|\leq |t|^c$ for all $i$. Since $\bar{V}\subseteq \Aff^N_k$ is the Zarisky closure of $\{b_{n_i}\}$ over $k$, we have $\{b_{n_i}\}\times_{\Spec k}\Spec(K^{\flat,\circ}/(t^s))$ is dense in $\eta(\bar{V})\times_{\Spec k}\Spec(K^{\flat,\circ}/(t^s))$ for any $s>0$. Therefore, we have $|h_c(b_{t+jl})|\leq |t|^c$ for all $j$. Considering $h_c\in R^{\flat,perf}$, for the sake of the structure of our perfectoid space, we have $|h_c(b'_{t+jl})|\leq |t|^c$ where $b'_{t+jl}$ is
\[
\varphi^{\flat}((b_{t+(j-k)l},b_{t+(j-k)l+1},\ldots)).
\]
In particular, we have $|g_c(\rho\circ\varphi^{-1}\circ\chi_{\{a_i\}}\circ\iota(a_{t+jl}))|\leq |t|^c$ which implies
\[
|H_j-g_c^\sharp(\varphi^{-1}\circ\chi_{\{a_i\}}\circ\iota(a_{t+jl}))|\leq |\varpi|^{1/2+c}.
\]
Since $c$ is arbitrary, we can conclude that $a_{t+jl}\in V$ for all $j$. Since $\{a_{t+jl}\}$ is an infinite set dense in $V$, we conclude that $V$ is periodic.
\end{proof}
For the rest of this section, we are going to give a new argument to show that the condition of backward orbit is as strong as the condition of coherent backward orbit in some cases. 

Let $F=(f_1,\ldots, f_N)$ and $a=(x_1,\ldots, x_n)\in \bar{\Q}^N$. We want $\Gal(f_i^n=x_i)\cap\Gal(f_j^n=x_j)$ as small as possible for all $i\neq j$. Then, we can move the preimage of $a$ to where we want, especially we want to move them such that they form a coherent backward orbit.
\begin{definition}
Let $f:\bar{\Q}\to\bar{\Q}$ be a polynomial. We say $f$ is stable if $f^n$ is irreducible for all $n$, and $f$ is eventually stable if the number of irreducible factors of $f^n$ is bounded as $n$ tends to infinity.
\end{definition}
We can replace this definition by using the "Galois orbits". Let $x\in \bar{K}^N$, and we call the set
\[
\{\sigma(x)=(\sigma(x_1),\ldots,\sigma(x_N))\mid \sigma\in\Gal(K^{\text{sep}}/K)\}
\]
the \textbf{Galois orbit} of $x$ for the Galois group $\Gal(K^{\text{sep}}/K)$ where $K^{\text{sep}}$ is the separable closure of $K$. Clearly, this is an equivalent relation, and then gives a partition of $F^{-n}(x)$, the $n$th preimage of $x$ for the function $F$. Obviously, the Galois orbits gives an equivalent relation $\sim$ on the set $F^{-n}(x)$, so $\#F^{-n}(x)/\sim$ is the number of Galois orbits.
\begin{definition}\label{eventuallystable}
We say $F$ is stable at $x$ if the number of Galois orbits of $F^{-n}(x)$ is $1$ for all $n\in \N$, and say $F$ is eventually stable at $x$ if the number of Galois orbits of $F^{-n}(x)$ is bounded for $n\in \N$.
\end{definition}
We find an argument which shows that the coherent condition is not really stronger if the function $F$ is assumed to be eventually stable.
\begin{theorem}\label{finialthm}
Let $F=(f_1,\ldots,f_N):K^N\to K^N$ where $f_i\in K[X_1,\ldots X_N]$ be eventually stable, and $V$ be an variety defined over a finite extension of $K$ where the separable algebraic closure $K^{\text{sep}}$ of $K$ is not finite over $K$ . If 
we have a backward orbit of $x_0$ that has infinite intersection with the variety $V$, then we can find a coherent backward orbit $\{a_i\}$ such that $\{a_i\}\cap V$ is infinite.
\end{theorem}
\begin{proof}
Let $L$ be the defining field of $V$ which is finite over $K$, so the $\Gal(K^\text{sep}/L)$, a subgroup of $\Gal(K^\text{sep}/K)$ is infinite. Let $I\coloneqq\{n\in\N\mid F^{-n}(x_0)\cap V\neq \emptyset\}$. Since $F$ is eventually stable, for any $n$, $F^{-n}(x_0)$ has at most $M$ many Galois orbits for the Galois group $\Gal(K^{\text{sep}}/L)$. Note that once we have $x\in V$, then the whole Galois orbit of $x$ lives on $V$, i.e. $V\cap\{\sigma(x)\mid\sigma\in\Gal(K^\text{sep}/L)\}=\{\sigma(x)\mid\sigma\in\Gal(K^\text{sep}/L)\}$ or $\emptyset$.

We will construct the coherent orbit $\{a_i\}$ by induction. For $i=0$, we can simply let $a_i=x_0$ and $n_0=0$.

For $i\geq 1$, since we have a backward orbit that passes through the variety $V$ infinitely many times, there exists $n_{i}>n_{i-1}$ with $F^{-n_{i}}(x_0)\cap V\neq\emptyset$. Find $a_i\in F^{-n_{i}}(x_0)$ which satisfies $F^{-n}(a_i)\cap V\neq\emptyset$ for infinitely many $n$ and $F^{n_{i+1}-n_i}(a_i)=a_{i-1}$. If we cannot find such $a_i$ in $F^{-n_i}(x_0)$, we will need to find a larger $n_i$ and repeat the process again. Since the function $F$ is eventually stable and the assumption that there exists a backward orbit of $x_0$ that intercepts the variety $V$ for infinitely many times, we can definitely find $a_i$ by replacing $n_i$ finitely many times. This construction guarantees that the sequence $\{a_i\}$ is a coherent backward orbit of $x_0$. 
\end{proof}
\begin{remark}
We can translate the induction step to the following pseudocode program. Although this algorithm loops endlessly, which means that the algorithm is not effective, it gives a basic idea about why the induction works.
\begin{algorithm}
    \SetKwInOut{Input}{Input}
    \SetKwInOut{Output}{Output}

    \underline{Coherent backward orbit} $(x_0,F,V)$\;
    \Input{An algebraic element $x_0$, a function $F$ and a variety $V$}
    \Output{A coherent orbit of $x_0$}
    $i\coloneqq 0$\;
    $n_0\coloneqq 0$\;
    $t\coloneqq 0$\;
    $L\coloneqq []$\;
    \Do{$i\geq 0$}
      {
      	check$\coloneqq0$\;
        Search $n_{i+1}>n_i$ with $F^{-(n_{i+1}-n_i)}(x_i)\cap V\neq\emptyset$\;
        $S\coloneqq F^{-(n_{i+1}-n_i)}(x_{n_i})\cap V$\;
        \Do{$F^{-n}(x_{n_{i+1}})\cap V=\emptyset$ for infinite but finite $n$}
        {
        	\If{$t>0$}
        	{
           	$S\coloneqq S\setminus\{\sigma(x_{i+1})\mid\sigma\in\Gal(K^{\text{sep}}/L)\}$\;
            \If{$S$ is empty}
            	{
                check$\coloneqq 1$\;
                break\;
                }
        	}
         	Randomly choose $x_{n_{i+1}}\in S$\;
            $t\coloneqq t+1$\;
		}
      \If{check$==1$}{$t\coloneqq 0$\;break\;}
      Add the element $x_{n_{i+1}}$ to the list $L$\;
      $t\coloneqq 0$\;
      $i\coloneqq i+1$\;
      }
      {
        return $L$\;
      }
    \caption{An "algorithm" for finding a coherent backward orbit of $x_0$ for $F$}
\end{algorithm}

The assumption that there exists a backward orbit of $x_0$ passing through the variety $V$ infinitely many times make sure the 8th line can be executed, and the 11st line to the 17th line actually will only execute finite many times because eventually $F$ will be stable at $x_{i+1}$.
\end{remark}

\begin{remark}
In \cite{GTZ2011}, authors gives a counterexample, but the polynomial in their paper is actually not eventually stable.\end{remark}

\bibliographystyle{amsalpha}
\bibliography{Perfected}

\end{document}